\documentclass[12pt]{amsart}

\usepackage{amssymb,latexsym,tikz}
\usepackage[all]{xy}
\usepackage[top=35mm, bottom=25mm, left=30mm, right=30mm]{geometry}

\begin{document}
\title[Generalized letterplace ideals]
{Poset ideals of $P$-partitions and generalized
letterplace and determinantal ideals}

\author{Gunnar Fl{\o}ystad}
\address{Matematisk institutt\\
        Universitetet i Bergen\\ 
        Postboks 7803\\
        5020 Bergen} 
\email{gunnar@mi.uib.no}


\date{\today}


\theoremstyle{plain}
\newtheorem{theorem}{Theorem}[section]
\newtheorem{corollary}[theorem]{Corollary}
\newtheorem*{main}{Main Theorem}
\newtheorem{lemma}[theorem]{Lemma}
\newtheorem{proposition}[theorem]{Proposition}
\newtheorem{conjecture}[theorem]{Conjecture}

\theoremstyle{definition}
\newtheorem{definition}[theorem]{Definition}
\newtheorem{fact}[theorem]{Fact}

\theoremstyle{remark}
\newtheorem{notation}[theorem]{Notation}
\newtheorem{remark}[theorem]{Remark}
\newtheorem{example}[theorem]{Example}
\newtheorem{claim}{Claim}


\newcommand{\psp}[1]{{{\bf P}^{#1}}}
\newcommand{\psr}[1]{{\bf P}(#1)}
\newcommand{\op}{{\mathcal O}}
\newcommand{\opw}{\op_{\psr{W}}}
\newcommand{\go}{\op}

\newcommand{\ini}[1]{\text{in}(#1)}
\newcommand{\gin}[1]{\text{gin}(#1)}
\newcommand{\kr}{{\Bbbk}}
\newcommand{\pd}{\partial}
\newcommand{\vardel}{\partial}
\renewcommand{\tt}{{\bf t}}


\newcommand{\coh}{{{\text{{\rm coh}}}}}


\newcommand{\modv}[1]{{#1}\text{-{mod}}}
\newcommand{\modstab}[1]{{#1}-\underline{\text{mod}}}

\newcommand{\sut}{{}^{\tau}}
\newcommand{\sumit}{{}^{-\tau}}
\newcommand{\til}{\thicksim}

\newcommand{\totp}{\text{Tot}^{\prod}}
\newcommand{\dsum}{\bigoplus}
\newcommand{\dprod}{\prod}
\newcommand{\lsum}{\oplus}
\newcommand{\lprod}{\Pi}

\newcommand{\La}{{\Lambda}}

\newcommand{\sirstj}{\circledast}

\newcommand{\she}{\EuScript{S}\text{h}}
\newcommand{\cm}{\EuScript{CM}}
\newcommand{\cmd}{\EuScript{CM}^\dagger}
\newcommand{\cmri}{\EuScript{CM}^\circ}
\newcommand{\cler}{\EuScript{CL}}
\newcommand{\clerd}{\EuScript{CL}^\dagger}
\newcommand{\clerri}{\EuScript{CL}^\circ}
\newcommand{\gor}{\EuScript{G}}
\newcommand{\cF}{\mathcal{F}}
\newcommand{\cG}{\mathcal{G}}
\newcommand{\cH}{\mathcal{H}}
\newcommand{\cM}{\mathcal{M}}
\newcommand{\cE}{\mathcal{E}}
\newcommand{\cD}{\mathcal{D}}
\newcommand{\cI}{\mathcal{I}}
\newcommand{\cP}{\mathcal{P}}
\newcommand{\cK}{\mathcal{K}}
\newcommand{\cL}{\mathcal{L}}
\newcommand{\cS}{\mathcal{S}}
\newcommand{\cC}{\mathcal{C}}
\newcommand{\cO}{\mathcal{O}}
\newcommand{\cJ}{\mathcal{J}}
\newcommand{\cU}{\mathcal{U}}
\newcommand{\cR}{\mathcal{R}}
\newcommand{\cQ}{\mathcal{Q}}
\newcommand{\mm}{\mathfrak{m}}

\newcommand{\dlim} {\varinjlim}
\newcommand{\ilim} {\varprojlim}

\newcommand{\CM}{\text{CM}}
\newcommand{\Mon}{\text{Mon}}


\newcommand{\Kom}{\text{Kom}}


\newcommand{\EH}{{\mathbf H}}
\newcommand{\res}{\text{res}}
\newcommand{\Hom}{\text{Hom}}
\newcommand{\inhom}{{\underline{\text{Hom}}}}
\newcommand{\Ext}{\text{Ext}}
\newcommand{\Tor}{\text{Tor}}
\newcommand{\ghom}{\mathcal{H}om}
\newcommand{\gext}{\mathcal{E}xt}
\newcommand{\id}{\text{{id}}}
\newcommand{\im}{\text{im}\,}
\newcommand{\codim} {\text{codim}\,}
\newcommand{\resol}{\text{resol}\,}
\newcommand{\rank}{\text{rank}\,}
\newcommand{\lpd}{\text{lpd}\,}
\newcommand{\coker}{\text{coker}\,}
\newcommand{\supp}{\text{supp}\,}
\newcommand{\Ad}{A_\cdot}
\newcommand{\Bd}{B_\cdot}
\newcommand{\Fd}{F_\cdot}
\newcommand{\Gd}{G_\cdot}


\newcommand{\sus}{\subseteq}
\newcommand{\sups}{\supseteq}
\newcommand{\pil}{\rightarrow}
\newcommand{\vpil}{\leftarrow}
\newcommand{\rpil}{\leftarrow}
\newcommand{\lpil}{\longrightarrow}
\newcommand{\inpil}{\hookrightarrow}
\newcommand{\pils}{\twoheadrightarrow}
\newcommand{\projpil}{\dashrightarrow}
\newcommand{\dotpil}{\dashrightarrow}
\newcommand{\adj}[2]{\overset{#1}{\underset{#2}{\rightleftarrows}}}
\newcommand{\mto}[1]{\stackrel{#1}\longrightarrow}
\newcommand{\vmto}[1]{\stackrel{#1}\longleftarrow}
\newcommand{\mtoelm}[1]{\stackrel{#1}\mapsto}
\newcommand{\bihom}[2]{\overset{#1}{\underset{#2}{\rightleftarrows}}}
\newcommand{\eqv}{\Leftrightarrow}
\newcommand{\impl}{\Rightarrow}

\newcommand{\iso}{\cong}
\newcommand{\te}{\otimes}
\newcommand{\into}[1]{\hookrightarrow{#1}}
\newcommand{\ekv}{\Leftrightarrow}
\newcommand{\equi}{\simeq}
\newcommand{\isopil}{\overset{\cong}{\lpil}}
\newcommand{\equipil}{\overset{\equi}{\lpil}}
\newcommand{\ispil}{\isopil}
\newcommand{\vvi}{\langle}
\newcommand{\hvi}{\rangle}
\newcommand{\susneq}{\subsetneq}
\newcommand{\sgn}{\text{sign}}


\newcommand{\xd}{\check{x}}
\newcommand{\ortog}{\bot}
\newcommand{\tL}{\tilde{L}}
\newcommand{\tM}{\tilde{M}}
\newcommand{\tH}{\tilde{H}}
\newcommand{\tvH}{\widetilde{H}}
\newcommand{\tvh}{\widetilde{h}}
\newcommand{\tV}{\tilde{V}}
\newcommand{\tS}{\tilde{S}}
\newcommand{\tT}{\tilde{T}}
\newcommand{\tR}{\tilde{R}}
\newcommand{\tf}{\tilde{f}}
\newcommand{\ts}{\tilde{s}}
\newcommand{\tp}{\tilde{p}}
\newcommand{\tr}{\tilde{r}}
\newcommand{\tfst}{\tilde{f}_*}
\newcommand{\empt}{\emptyset}
\newcommand{\bfa}{{\mathbf a}}
\newcommand{\bfb}{{\mathbf b}}
\newcommand{\bfd}{{\mathbf d}}
\newcommand{\bfl}{{\mathbf \ell}}
\newcommand{\bfx}{{\mathbf x}}
\newcommand{\bfm}{{\mathbf m}}
\newcommand{\bfv}{{\mathbf v}}
\newcommand{\bft}{{\mathbf t}}
\newcommand{\la}{\lambda}
\newcommand{\bfen}{{\mathbf 1}}
\newcommand{\bfe}{{\mathbf 1}}
\newcommand{\ep}{\epsilon}
\newcommand{\en}{r}
\newcommand{\tu}{s}
\newcommand{\Sym}{\text{Sym}}

\newcommand{\ome}{\omega_E}

\newcommand{\bevis}{{\bf Proof. }}
\newcommand{\demofin}{\qed \vskip 3.5mm}
\newcommand{\nyp}[1]{\noindent {\bf (#1)}}
\newcommand{\demo}{{\it Proof. }}
\newcommand{\demodone}{\demofin}
\newcommand{\parg}{{\vskip 2mm \addtocounter{theorem}{1}  
                   \noindent {\bf \thetheorem .} \hskip 1.5mm }}

\newcommand{\lcm}{{\text{lcm}}}


\newcommand{\dl}{\Delta}
\newcommand{\cdel}{{C\Delta}}
\newcommand{\cdelp}{{C\Delta^{\prime}}}
\newcommand{\dlst}{\Delta^*}
\newcommand{\Sdl}{{\mathcal S}_{\dl}}
\newcommand{\lk}{\text{lk}}
\newcommand{\lkd}{\lk_\Delta}
\newcommand{\lkp}[2]{\lk_{#1} {#2}}
\newcommand{\del}{\Delta}
\newcommand{\delr}{\Delta_{-R}}
\newcommand{\dd}{{\dim \del}}
\newcommand{\Del}{\Delta}

\renewcommand{\aa}{{\bf a}}
\newcommand{\bb}{{\bf b}}
\newcommand{\cc}{{\bf c}}
\newcommand{\xx}{{\bf x}}
\newcommand{\yy}{{\bf y}}
\newcommand{\zz}{{\bf z}}
\newcommand{\mv}{{\xx^{\aa_v}}}
\newcommand{\mF}{{\xx^{\aa_F}}}

\newcommand{\Symm}{\text{Sym}}
\newcommand{\pnm}{{\bf P}^{n-1}}
\newcommand{\opnm}{{\go_{\pnm}}}
\newcommand{\ompnm}{\omega_{\pnm}}

\newcommand{\pn}{{\bf P}^n}
\newcommand{\hele}{{\mathbb Z}}
\newcommand{\nat}{{\mathbb N}}
\newcommand{\rasj}{{\mathbb Q}}
\newcommand{\bfone}{{\mathbf 1}}

\newcommand{\dt}{\bullet}
\newcommand{\disk}{\scriptscriptstyle{\bullet}}

\newcommand{\cxF}{F_\dt}
\newcommand{\pol}{f}

\newcommand{\Rn}{{\mathbb R}^n}
\newcommand{\An}{{\mathbb A}^n}
\newcommand{\frg}{\mathfrak{g}}
\newcommand{\PW}{{\mathbb P}(W)}

\newcommand{\pos}{{\mathcal Pos}}
\newcommand{\g}{{\gamma}}

\newcommand{\Vaa}{V_0}
\newcommand{\Bp}{B^\prime}
\newcommand{\Bpp}{B^{\prime \prime}}
\newcommand{\bbp}{\mathbf{b}^\prime}
\newcommand{\bbpp}{\mathbf{b}^{\prime \prime}}
\newcommand{\bp}{{b}^\prime}
\newcommand{\bpp}{{b}^{\prime \prime}}

\newcommand{\oLa}{\overline{\Lambda}}
\newcommand{\ov}[1]{\overline{#1}}
\newcommand{\ovv}[1]{\overline{\overline{#1}}}
\newcommand{\tm}{\tilde{m}}
\newcommand{\po}{\bullet}

\def\CC{{\mathbb C}}
\def\GG{{\mathbb G}}
\def\ZZ{{\mathbb Z}}
\def\NN{{\mathbb N}}
\def\RR{{\mathbb R}}
\def\OO{{\mathbb O}}
\def\QQ{{\mathbb Q}}
\def\VV{{\mathbb V}}
\def\PP{{\mathbb P}}
\def\EE{{\mathbb E}}
\def\FF{{\mathbb F}}
\def\AA{{\mathbb A}}

\renewcommand{\op}{{\text{op}}}

\begin{abstract}
For any finite poset $P$ we have the poset of isotone maps
$\Hom(P,\NN)$, also called $P^{\op}$-partitions. 
To any poset ideal $\cJ$ in $\Hom(P,\NN)$, finite or infinite,
we associate monomial
ideals: the letterplace ideal $L(\cJ,P)$ and the Alexander dual co-letterplace
ideal $L(P,\cJ)$, and study them. We derive a class of 
monomial ideals in $\kr[x_p, p \in P]$ called $P$-stable. 
When $P$ is a chain we establish a duality
on strongly stable ideals. 
We study the case when $\cJ$ is a principal poset ideal. 
When $P$ is a chain we construct
a new class of determinantal ideals which generalizes ideals of 
{\it maximal} minors
and whose initial ideals are letterplace ideals of principal poset ideals.
\end{abstract}

\maketitle

\newcommand{\Supp}{\text{Supp}}
\newcommand{\hull}{\text{hull}}
\renewcommand{\cF}{{\mathcal F}}
\newcommand{\fp}{{\mathfrak p}}
\newcommand{\oI}{\overline{I}}
\newcommand{\hI}{\hat{I}}

\section*{Introduction}
For a finite poset $P$ the isotone maps $P \pil \NN$ form a 
poset $\Hom(P,\NN)$ where two maps $\phi \leq \psi$ if 
$\phi(p) \leq \psi(p)$ for every $p \in P$. Denoting by $P^{\op}$ the
opposite poset, $\Hom(P,\NN)$ identifies with {\it $P^{\op}$-partitions},
originally introduced and studied by R.Stanley in the classic \cite{StaPP},
see also \cite[Section 4.5]{StaEnu}. 
(For a poset $Q$ a $Q$-partition is a map $Q \pil \NN$ such that
$q_1 \leq q_2$ implies $\phi(q_1) \geq \phi(q_2)$.)
In \cite{FGH} the author together with B.Greve and J.Herzog introduced
letterplace and co-letterplace ideals, monomial ideals in a polynomial
ring, associated to poset ideals in 
$\Hom(P,[n])$ where $[n] = \{1 < 2 < \cdots < n \}$ is the chain. These
identify naturally as the finite poset ideals in $\Hom(P,\NN)$.  
Here we generalize this by associating letterplace and 
co-letterplace ideals to any poset ideal $\cJ$ of $\Hom(P, \NN)$ or,
alternatively formulated, to any poset ideal in the poset of 
$P^{\op}$-partitions. 

\noindent{\bf Letterplace and co-letterplace ideals.} To any isotone map 
$\phi : P \pil \NN$ there is associated two natural subsets of $P \times \NN$. 
The first is the graph $\Gamma \phi$, the second we call the {\it ascent} 
$\Lambda \phi$ and is
$\{ (p,i) \, | \, \phi(q) \leq i < \phi(p) \text{ for every } q < p \}$. Denote by 
$\kr[x_{P \times \NN}]$ the polynomial ring in the infinite number
of variables $x_{p,i}$ where $p \in P$ and $i \in \NN$. The co-letterplace ideal
associated to a finite poset ideal $\cJ \sus \Hom(P,\NN)$ is a monomial ideal
$L(P,\cJ)$ in $\kr[x_{P \times \NN}]$ whose generators are given by monomials 
associated to
graphs $\Gamma \phi \sus P \times \NN$  for $\phi$ in $\cJ$.
For a general poset ideal $\cJ$ the generators correspond to graphs 
of certain isotone maps $\phi : I \pil \NN$ where the $I$ are poset ideals 
in $P$. The letterplace ideal 
$L(\cJ,P)$ in $\kr[x_{P \times \NN}]$ has generators given by the
ascents $\Lambda \phi$ of $\phi$'s in the complement filter $\cJ^c$
of $\Hom(P,\NN)$. These monomial ideals are Alexander dual,
Proposition \ref{pro:defiAlexdual}. Furthermore these ideals have finite
minimal generating sets, so if we let $S$ be the set of all variables 
occurring in these generators, the support of these ideals,
we can consider these monomial ideals to be in a finitely
generated polynomial ring $\kr[x_S]$. 

\noindent{\bf Regular quotients.} If $S \sus P \times \NN$ is the support, 
an isotone map $\psi : S \pil R$ gives a map of polynomial rings
$\kr[x_S] \pil \kr[x_R]$ given by dividing $\kr[x_S]$ by a regular sequence
consisting of differences $x_s - x_{s^\prime}$ where $\phi(s) = \phi(s^\prime)$.
We give conditions, Theorems \ref{thm:regLP} and \ref{thm:regCOLP}, 
such that this is a regular sequence for $\kr[x_S]/L(\cJ,P)$ resp.
$\kr[x_S]/L(P,\cJ)$. This generalizes the main results of \cite{FGH} 
where we by many examples show the omnipresence of letter- and co-letterplace
ideals in the literature on monomial ideals.

\noindent{\bf $P$-stable ideals in $\kr[x_P]$.} In particular the projection
$S \sus P \times \NN \mto{p} P$ fulfill these conditions for the letterplace
ideal $L(\cJ,P)$. Hence we get ideals $L^p(\cJ,P) \sus \kr[x_P]$ 
whose quotient ring is a regular quotient of $\kr[x_S]/L(\cJ,P)$. 
When $P$ is the antichain these give  all the monomial ideals
in $\kr[x_P]$, and $L(\cJ,P)$ is the standard polarization of the monomial
ideal $L^p(\cJ,P) \sus \kr[x_P]$. 
When $P = [m]$ is the chain, these give precisely the 
strongly stable ideals in $\kr[x_{[m]}]$. For general $P$ we call such ideals 
$P$-stable and we give conditions for ideals to fulfill this notion.
The ideals $L(\cJ,P)$ will then be non-standard polarizations of
$L^p(\cJ,P)$. 
The notion of $P$-stable is not
the same as the notion of $P$-Borel in \cite{FMS}. Rather the notion
of $P$-stable is more subtle. For instance a power of the graded maximal
ideal $\mm^n$ in $\kr[x_P]$, with $n \geq 2$, is $P$-stable 
iff $P$ is the disjoint union of rooted trees with the roots on top.

\noindent{\bf Strongly stable ideals.} In case $P$ is the chain, we show the Alexander
duality of $L(\cJ,P)$ and $L(P,\cJ)$ induces a duality between strongly
stable ideals in $\kr[x_1, \ldots, x_m]$ and finitely generated $m$-regular 
strongly stable ideals in $\kr[x_{\NN}]$. In particular it gives a duality 
between $n$-regular strongly stable ideals in $\kr[x_1, \ldots, x_m]$ and
$m$-regular strongly stable ideals in $\kr[x_1, \ldots, x_n]$, which
seemingly has not been noticed before. These are joint results with
Alessio D'Ali and Amin Nematbakhsh. 

\noindent{\bf Principal poset ideals.} A distinguished class of poset ideals of 
$\Hom(P,\NN)$ arises by considering an isotone map $\alpha : P \pil \NN$ and
letting $\cJ$ be the set of all $\phi$ such that $\phi \leq \alpha$.
These give the {\it principal} letterplace ideals $L(\alpha,P)$ and
co-letterplace ideals $L(P,\alpha)$. The ideals $L(n,P)$ and $L(P,n)$
of \cite{FGH} 
are the special cases when $\alpha$ is the constant map sending each 
$p$ to $n$. When $P$ is the antichain $L(\alpha,P)$ is a complete
intersection, and these ideals may therefore be considered as generalizations
of complete intersections.

\noindent{\bf Determinantal ideals.} A particular case of principal letterplace
ideals is when $P$ is the chain $[m]$. In \cite{FGH} we showed that
$L(n,[m])$ is the initial ideal of the ideal of maximal minors of an
$(m+n-1)\times n$-matrix. Here we show that $L(\alpha, [m])$ is also an
initial ideal of a class of determinantal ideals. This is a class which
seemingly has not been considered before, and which in a natural way
generalizes the determinantal ideals of {\it maximal minors}. All these
ideals are Cohen-Macaulay of codimension $m = |P|$.


\medskip
The organization of the paper is as follows. In Section \ref{sec:defi}
we define the (generalized) letterplace ideal $L(\cJ,P)$ and the co-letterplace
ideal $L(P,\cJ)$ and show they are Alexander dual.
Section \ref{sec:reg} gives the conditions ensuring that we divide
out by a regular sequence of variable differences for the quotient
rings.  Section \ref{sec:monXP}
gives the conditions for a monomial ideal in $\kr[x_P]$ to be $P$-stable.
Principal letterplace ideals are introduced in Section \ref{sec:fin}
and we show how the initial ideals of certain classes of determinantal
ideals studied by W.Bruns and U.Vetter in \cite{BrVe} and J.Herzog and
N.Trung in \cite{HeTr} are regular quotients of principal letterplace ideals.
Section \ref{sec:detmax} gives the
class of determinantal ideals generalizing ideals of maximal minors.
In Section \ref{sec:ststable} we show the duality on strongly stable ideals,
and the last Section \ref{sec:proof} gives the proofs of the statements
of Section \ref{sec:reg} on regular sequences.

\section{The poset of isotone maps to $\NN$}
\label{sec:defi}

We first give various notions and results concerning order preserving
maps from a poset $P$ to the integers $\NN$. Then we define letterplace
and co-letterplace ideals in a more general setting than in the previous
article \cite{FGH}.

\subsection{Graphs and bases of isotone maps}

If $P$ and $Q$ are posets, an isotone map
$\phi: P \pil Q$ is a map such that $p \leq p^\prime$ implies
$\phi(p) \leq \phi(p^\prime)$. The set of isotone map is denoted
$\Hom(P,Q)$ and is itself a poset with $\psi \leq \phi$ if 
$\psi(p) \leq \phi(p)$ for every $p \in P$. Thus the category of posets
has an internal Hom.
We also have the product poset $P \times Q$.
This makes the category of posets into a symmetric monoidal
closed category. A subset $J \sus Q$ is a {\it poset ideal} if 
$q \in J$ and $q^\prime \leq q$ implies $q^\prime \in J$. A subset 
$F \sus Q$ is a {\it poset filter} if $q \in F$ and $q \leq q^\prime$
implies $q^\prime \in F$. Note that $J$ is a poset ideal, iff its
complement $J^c$ is a poset filter.

Denote by $\NN = \{0,1,2,\cdots\}$ the natural numbers including $0$.
Henceforth we shall always assume we have a {\it finite} partially 
ordered set $P$. 
We get the poset $\Hom(P,\NN)$.

\begin{remark} $\Hom(P,\NN)$ is also a monoid coming from addition on $\NN$:
$\phi + \psi$ is the map sending $p$ to $\phi(p) + \psi(p)$. 
Denote by $\hat{P} = P \cup \{1\}$ where $1$ is a maximal element, so
$1 > p$ for any $p \in  P$. Then the semi-group ring  $\kr \Hom(\hat{P},\NN)$ 
is the Hibi ring
of the distributive lattice $D(P)$ of order ideals in $P$. See
V.Ene \cite{Ene} for a survey on Hibi rings. The semi-group ring 
$\kr \Hom(P,\NN)$ was first studied by A.Garcia in \cite{Garsia} and 
more recently by  V.Feray and V.Reiner in \cite{FeRe}. 
We shall in this article only 
use the poset structure on $\Hom(P,\NN)$ and not the monoid structure.
\end{remark}

\begin{lemma}[Generalized Dickson's lemma] \label{lem:defiFinmin}
For a finite poset $P$, 
any subset $S$ of $\Hom(P, \NN)$ has a finite set of minimal elements.
\end{lemma}

\begin{proof}
Let $q \in P$ be a maximal element. 
The inclusion  $P\backslash \{q \} \hookrightarrow P$ induces a
projection $\Hom(P,\NN) \pil \Hom(P \backslash \{q \}, \NN)$. Let
$T$ be the image of $S$ under this projection. By induction
$T$ has a finite set of minimal elements (maps): $\alpha_1, \ldots, \alpha_r$.
Consider the $\phi$ such that the restriction $\phi_{|P\backslash{q}} = \alpha_i$,
and let $n_i$ be the minimal value of $\phi(q)$ for these $\phi$.
Let $N$ be the maximum of
the $n_i$. Each $\alpha_i$ extends to a map $\beta_i \in S$ 
with $\beta_i(q) \leq N$.

For $j \leq N$ let $S_j \sus S$ be the set of all $\phi \in S$
such that $\phi(q) = j$, and let $T_j$ be its projection
onto $\Hom(P \backslash \{q \}, \NN)$. By induction, each $T_j$
for $j = 1, \ldots, N$ has a finite set of minimal elements $T_j^{\min}$.
Extending $T_j^{\min}$ naturally to $S$ we get a set $S_j^{\min} \sus S$

If now $\phi \in S$, then
either $\phi(q) \geq N$, and so $\phi \geq \beta_i$ for some
$\beta_i$, or $\phi(q) = j \leq N$ and then $\phi \geq \beta$
for some $\beta \in S_j^{\min}$. Thus $S$ has a finite set of minimal
elements.
\end{proof}
 
\begin{definition} 
Let $\phi : P \pil \NN$ be an isotone map. The {\it graph} of
$\phi$ is 
\[ \Gamma \phi = \{ (p,i) \, | \, \phi(p) = i \}. \]
The {\it ascent} of $\phi$ is
\[ \Lambda\phi = \{ (p,i) \, | \, \phi(q) \leq i < \phi(p) \text{ for all }
q < p \}.\]
Both the graph and the ascent are subsets of $P \times \NN$.
\end{definition} 

\begin{remark}
The union $\Gamma \phi \cup \Lambda \phi$ is denoted $T_* \phi$ in 
\cite[Section2]{DFN-COLP}. The intervals $[\Gamma \phi, T_* \phi]$ index
the resolution of the co-letterplace ideals $L(P,n;\cJ)$, see 
\cite[Subsection 4.3]{DFN-COLP}. This generalizes the Eliahou-Kervaire
resolutions of strongly stable ideals generated in a single degree.
\end{remark}

\begin{lemma}
Let $\cF \sus \Hom(P, \NN)$ be a filter. 
The set of all ascents $\Lambda\phi$ where
$\phi \in \cF$, has a finite set of inclusion minimal elements, i.e.
which are minimal for the partial order of inclusion on sets.
\end{lemma}

\begin{proof}
Let $N$ be the maximum of all values $\phi(p)$ as $\phi$ ranges 
over the minimal elements of $\cF$ (which is a finite set by 
the previous Lemma \ref{lem:defiFinmin}) and $p$ ranges over $P$. 
Given an isotone map $\psi$ in $\cF$, let 
\[ n = \max \psi = \max \{ \psi(p) \, | \, p \in P \}. \]

If $n > N$ we want to show that $\Lambda \psi$ is not inclusion
minimal among $\Lambda \phi$ where $\phi \in \cF$. 
Then any inclusion  minimal $\Lambda \phi$ for $\phi \in \cF$ will have
$\phi(p) \leq N$ for all $p \in P$, and so there is only a finite
number of such $\phi$.

So suppose $n > N$. Then $\psi$ is not minimal in $\cF$.
Let $F \sus P$ be the poset filter of all $p \in P$ with $\psi(p) = n$, and
$F_{\min}$ the minimal elements in $F$. Then the map

\[ \phi(p) = \begin{cases} \psi(p) & p \not \in F \\
                                 n-1 & p \in F
                  \end{cases}, \]
is also in $\cF$.
But then clearly $\Lambda \phi = \Lambda \psi \backslash 
\{ (p, n -1) \, | \, p \in F_{\min} \}$. This proves the statement.
\end{proof}

\begin{definition}
Let $\cJ \sus \Hom(P, \NN)$ be a poset ideal. A {\it marker} for $\cJ$
is given by a poset ideal $I \sus P$ and an isotone map 
$\alpha : I \pil \NN$ such that {\it every} isotone map 
$\phi : P \pil \NN$ with 
restriction $\phi_{|I} = \alpha$, is in $\cJ$. It is a {\it minimal marker}
if no restriction $\alpha_{|J}$ to a poset ideal $J$
properly contained in $I$, is a marker for $\cJ$. This means that
the graph of the marker  $\alpha$ 
is inclusion  minimal among graphs of markers for $\cJ$.
\end{definition} 

\begin{remark} When $P$ is an antichain, 
the notion of poset ideal in $\Hom(P, \NN)$ is close to
the notion of multicomplex studied in \cite[Section 9]{HePo}.
\end{remark}

\begin{proposition} Let $\cJ \sus \Hom(P, \NN)$ be a poset ideal. 
The (unique) set of minimal markers is a finite set.
\end{proposition}

\begin{proof}
We first consider minimal markers $\phi: P \pil \NN$, i.e. minimal
markers whose domain is $P$. 
Let $\cM$ be the set of such minimal markers.
Suppose $\cM$ is infinite. Then there is some maximal $q \in P$ 
such that $\phi(q)$ may become arbitrarily large for $\phi \in \cM$. 
If the set $\cM_q$ of restrictions $\phi_{| P \backslash \{q \}}$ where $\phi \in \cM$
is finite, then clearly one of these restricted maps is a marker,
contradicting that
$\phi$ is a minimal marker. So the set $\cM_q$ is infinite. But by the 
above Proposition \ref{lem:defiFinmin}
it contains a finite set of minimal elements. Then for at least
one of these minimal elements, $\alpha$, there is an infinite number of 
$\phi \in \cM$ such that $\phi_{| P \backslash \{q \}} \geq \alpha$ and
$\phi(q)$ is arbitrarily large.
But since $\cJ$ is a poset ideal, there must then be $\phi : P \pil \NN$ with 
$\phi_{| P \backslash \{q \}} =  \alpha$ and $\phi(q)$ arbitrarily large, 
contradicting that $\cM$ consists of minimal markers. 

  Hence the set of inclusion minimal markers whose domain is $P$, is
finite. Let $I$ be a poset ideal contained in $P$.
Let $\cJ_{|I}$ be the poset ideal of $\Hom(I,\NN)$ consisting of 
markers supported on $I$. 
By the argument as above, there is a finite number of minimal markers
whose domain is $I$. Since there is only a finite number of such $I$
there will be only a finite number of minimal markers for $\cJ$.
\end{proof}

\subsection{Letterplace and co-letterplace ideals}

For a set $S$ let $\kr[x_S]$ be the polynomial ring over a field $\kr$ 
in the variables
$x_s$ where $s \in S$. We shall in particular consider the polynomial
ring $\kr[x_{P \times \NN}]$, a polynomial ring in the infinite number
of variables $x_{p,i}$ where $p \in P, i \in \NN$. 

\begin{definition} \label{def:defiLP}
Let $\cJ \sus \Hom(P,\NN)$ be a poset ideal.
The {\it co-letterplace ideal} $L(P,\cJ)$ is the monomial ideal
in $\kr[x_{P \times \NN}]$ generated by the monomials 
$m_{\Gamma \alpha}$ associated to the graphs
of markers $\alpha$ of $\cJ$.
It is clearly the same as the ideal generated by the finite set of monomials 
associated to graphs of minimal markers.

The {\it letterplace ideal} $L(\cJ,P)$ is the monomial ideal in 
$\kr[x_{P \times \NN}]$ generated by the monomials $m_{\Lambda\phi}$ associated to
ascents of the isotone maps $\phi$ in the complement filter $\cJ^c$. 
It is clearly the same as the ideal generated by the finite set of monomials of 
inclusion minimal ascents.
\end{definition}

\begin{remark} Let $[n] = \{1 < 2 < \cdots < n \}$.
In \cite{FGH} the setting was a poset ideal $\cJ \sus \Hom(P,[n])$
and we defined letterplace and co-letterplace ideals 
$L(n,P;\cJ)$ and $L(P,n;\cJ)$ in the polynomial ring $\kr[x_{P \times [n]}]$. 
This corresponds to finite poset ideals
$\cJ$ in the definition  above, see Section \ref{sec:fin}.
In \cite{DFN-COLP} we showed that the Stanley-Reisner ideal 
$L(n,P;\cJ)$ defines a simplicial ball
and described precisely the Gorenstein ideal defining its boundary.
\end{remark} 

Recall that for two squarefree monomial  ideals $I$ and $J$ in a
polynomial ring $S$ then $I$ is said to be {\it Alexander dual} to $J$  if 
the monomials in $I$ are precisely 
the monomials in $S$ which have nontrivial common 
divisor with every monomial in $J$. Then $J$ will also be Alexander dual to 
$I$. 

The following generalizes \cite[Theorem 1.1]{EHM} and \cite[Prop.1.2]{FGH}.

\begin{proposition} \label{pro:defiAlexdual}
The letterplace ideal $L(\cJ,P)$ and the co-letterplace ideal $L(P,\cJ)$
are Alexander dual.
\end{proposition}

Our proof is close to the proof of Theorem 5.9 of \cite{FGH}.
Let us first recall some lemmata.
Let $\cJ$ be a poset ideal in $\Hom(P,\NN)$ and $\cJ^c$ its complement filter.
The following is Lemma 5.7 in \cite{FGH}.

\begin{lemma} \label{lem:DefiGaLa}
Let $\phi \in \cJ$ and $\psi \in \cJ^c$. Then $\Gamma \phi \cap \Lambda\psi$ 
is nonempty.
\end{lemma}

The following is Lemma 5.8 in \cite{FGH}.
\begin{lemma} \label{lem:DefiSLa}
Given a subset $S$ of $P \times \NN$. Suppose it is disjoint from $\Gamma \phi$
for some $\phi$ in $\Hom(P, \NN)$. If $\phi$ is minimal such w.r.t.
the partial order on $\Hom(P, \NN)$, then $S \supseteq \Lambda\phi$.
\end{lemma}

\begin{proof}[Proof of Proposition \ref{pro:defiAlexdual}.]
We show the following:

\noindent 1. The letterplace ideal $L(\cJ,P)$ is contained in the Alexander dual of
$L(P,\cJ)$: Every monomial in $L(\cJ,P)$ has non-trivial common divisor
with every monomial in $L(P,\cJ)$.

\noindent 2. The Alexander dual of $L(P,\cJ)$ is contained in the 
letterplace ideal $L(\cJ,P)$:
If a finite $S \sus P \times \NN$ intersects every 
$\Gamma \alpha$ where $\alpha : I \pil \NN$ is a marker for $\cJ$, 
the monomial $m_S$ is in $L(\cJ,P)$.

\medskip
\noindent 1. Let $\alpha : I \pil \NN$ be a marker for $\cJ$ and
let $\psi \in \cJ^c$.
By Lemma \ref{lem:DefiGaLa}, for every extension $\phi$ of $\alpha$,
$\Gamma \phi$ and $\Lambda \psi$ intersect nonempty. But then clearly
$\Gamma \alpha$ and $\Lambda \psi$ intersect nonempty.

\medskip
\noindent 2. Suppose a finite $S$ intersects every $\Gamma \alpha$ where 
$\alpha : I \pil \NN$ is a marker for $\cJ$. Then $S$ intersects every
$\Gamma \phi$ where $\phi \in \cJ$. Let $\psi$ be a minimal element in
$\Hom(P,\NN)$ such that $\Gamma \psi$ and $S$ intersect empty. Then
$\psi \in \cJ^c$ and by Lemma \ref{lem:DefiSLa} $S \supseteq \Lambda \psi$. 
Therefore the monomial $m_S$ is in $L(\cJ,P)$.  
\end{proof}

When $I$ and $J$ are two Alexander dual monomial ideals, their
set of minimal generators will involve exactly the same variables.
The {\it support} $\Supp(\cJ)$ of the poset ideal $\cJ$ of $\Hom(P,\NN)$ is
the set pairs $(p,i)$ such that $x_{p,i}$ is a variable in one of the 
minimal generators of $L(P,\cJ)$, or equivalently $L(\cJ,P)$.
This is a finite set by the minimality observations in Definition 
\ref{def:defiLP}.
We can then consider the letterplace and co-letterplace
ideals $L(P,\cJ)$ and $L(\cJ,P)$ to live in the ring $\kr[x_{\Supp(\cJ)}]$, 
or in any polynomial ring $\kr[x_S]$ where $\Supp(\cJ) \sus S \sus P \times
\NN$. Most of our statements will be independent of what the ambient ring is.
In general the set $\Supp(\cJ)$ is not so easily described, but in the
case when $\cJ$ is finite there is a nice description, Section 
\ref{sec:fin}.

\section{Regular quotients}
\label{sec:reg}

In \cite[Section 3]{FGH} we gave many examples of ideals which derive from
letterplace ideals, and in \cite[Section 6]{FGH} many examples of ideals which derive from
co-letterplace ideals. But the point is that they {\it derive} from 
them by cutting down the corresponding quotient ring by a regular sequence.
The conditions ensuring that we have a regular sequence, are the main
technical results of \cite{FGH}. Here we generalize these results
to the setting of any poset ideal in $\Hom(P, \NN)$. 

\subsection{Regular sequences}
A subset $S$ of $P \times \NN$ gets the induced poset structure. 
Let $S \mto{\phi} R$ be an isotone map.
Denote by $P^{\op}$ be the opposite poset of $P$, i.e. $p \leq_{\op} q$
if $p \geq q$. 
We say that $\phi$ has {\it right} strict chain fibers if $\phi^{-1}(r)
\sus S$ is a chain when considered in $P^{\op} \times \NN$, and all
elements in this chain have distinct second coordinate. 
So if $(p,i)$ and $(q,j)$ are
distinct elements in the fiber with $i \leq j$, then $i < j$ and $p \geq q$
for the partial order on $P$.

Given a poset ideal $\cJ \sus \Hom(P, \NN)$, choose a {\it finite} 
$S$ such that 
$\Supp(\cJ) \sus S \sus P \times \NN$.
Any map $S \mto{\phi} R$ gives a map of linear spaces 
$\langle x_S \rangle \mto{\tilde{\phi}} \langle x_R \rangle$. Let $B$ be a 
basis for
the kernel of $\tilde{\phi}$, whose elements are differences $x_{p,i} - x_{q,j}$
such that $\phi(p,i) = \phi(q,j)$. 

\begin{theorem} \label{thm:regLP}
Given a poset ideal $\cJ \sus \Hom(P, \NN)$,  and let 
$\Supp(\cJ) \sus S \sus P \times \NN$ with $S$ finite.  
Let $\phi : S \pil R$ be an isotone map with right strict chain fibers. 
Then the basis $B$ is a regular sequence for $\kr[x_S]/L(\cJ,P)$. 
\end{theorem}

We prove this in Section \ref{sec:proof}. It is a generalization of 
Theorem 5.6 in \cite{FGH}. Its proof follows rather closely the proof 
in \cite{FGH} but with some new elements.
We write $L^\phi(\cJ,P)$ for the ideal in $\kr[x_R]$ 
generated by the image of $L(\cJ,P)$ by $\phi$.

Consider again an isotone map $S \mto{\phi} R$ as above. We say it has
{\it left} strict chain fibers if $\phi^{-1}(r) \sus S$ is a chain 
when considered in $P^\op \times \NN$ and all its elements have distinct
first coordinates.

\begin{theorem} \label{thm:regCOLP}
Given a poset ideal $\cJ \sus \Hom(P, \NN)$, and let 
$\Supp(\cJ) \sus S \sus P \times \NN$ with $S$ finite.
Let $\phi : S \pil R$ be an isotone map with left strict chain fibers. 
Then the basis $B$ is a regular sequence for $\kr[x_S]/L(P,\cJ)$. 
\end{theorem}

Again we prove this in Section \ref{sec:proof}, and it is a generalization of 
Theorem 5.12 in \cite{FGH}. Its proof also follows rather closely the proof 
in \cite{FGH}.
We write $L^\phi(P,\cJ)$ for the ideal in $\kr[x_R]$ 
generated by the image of $L(P,\cJ)$ by $\phi$.

\section{Monomial ideals in $\kr[x_P]$}
\label{sec:monXP}

When $P$ is the antichain on $m$ elements, then $\Hom(P,\NN)$ identifies
as $\NN^m$. A poset filter $\cJ^c$ in $\NN^m$ corresponds naturally to 
a monomial ideal in $\kr[x_1, \ldots, x_m]$. In this case
the letterplace ideal $L(\cJ,P)$ is the standard polarization of this
monomial ideal.
  In this section we associate monomial ideals in $\kr[x_P]$ to poset filters
$\cJ^c \in \Hom(P,\NN)$ for any poset $P$. The extra structure added 
in this setting, is that the letterplace ideal
$L(\cJ,P)$ may be considered a (non-standard) polarization of the monomial
ideal in $\kr[x_P]$.

Let $p : P \times \NN \pil P$ be the projection map onto the first coordinate
$P$. This map has right strict chain fibers, and so for any poset 
ideal $\cJ \sus P \times \NN$ we get by Theorem \ref{thm:regLP} an ideal
$L^p(\cJ,P)$ in $\kr[x_P]$ whose quotient ring is a regular quotient of 
the quotient ring $\kr[x_S]/L(\cJ,P)$, for suitable finite
$S \sus P \times \NN$.

In the first subsection we get a direct description of the ideal 
$L^p(\cJ,P)$ and its quotient ring in terms of the poset ideal $\cJ$. 
In the next subsection we achieve an intrinsic description of which monomial
ideals in $\kr[x_P]$ that come from a poset ideal $\cJ$ in $\Hom(P,\NN)$.


\subsection{Correspondence between monomials in 
$\kr[x_P]$ and elements of $\Hom(P,\NN)$}
Given a poset $P$. The polynomial ring $\kr[x_P]$ may be identified
as the semi-group ring of the monoid $\NN P$ consisting of sums
$\sum_{p \in P} n_p p$ where $n_p \in \NN$. This monoid is naturally also
a poset by $\sum_{p \in P} n_p p
\leq \sum_{p \in P} m_p p$ if each $n_p \leq m_p$
(so this poset structure is determined only by the cardinality of $P$).
The monoid $\NN P$ naturally identifies as the monomials in 
$\kr[x_P]$ and we shall freely use this identification. 

Now given $\phi \in \Hom(P, \NN)$. The ascent 
$\Lambda \phi \sus P \times \NN$.
Taking the formal sum of its elements, we consider it as an element
of the monoid $\NN (P \times \NN)$. By the projection $P \times \NN \pil P$
we get the map of monoids $\NN(P \times \NN) \pil \NN P$. Denote by
$\overline{\Lambda}\phi$ the image of $\Lambda\phi$ by this map.

\begin{example}
Consider the poset $P$ below and the map $\phi : P \pil \NN$ sending
the vertices to the numbers to the right.
\begin{center}
\begin{tikzpicture}[scale=1, vertices/.style={draw, fill=black, circle, inner sep=1pt}]
              \node [vertices, label=right:{$a$}] (0) at (0,0){};
              \node [vertices, label=right:{$c$}] (1) at (0,1){};
              \node [vertices, label=right:{$b$}] (2) at (0.75,0){};
              \node [vertices, label=right:{$d$}] (3) at (0.75,1){};
              \foreach \to/\from in {0/1, 1/2, 2/3}
      \draw [-] (\to)--(\from);
\end{tikzpicture}
\hskip 5mm
$\mapsto$
\hskip 5mm
\begin{tikzpicture}[scale=1, vertices/.style={draw, fill=black, circle, inner sep=1pt}]
              \node [vertices, label=right:{$2$}] (0) at (0,0){};
              \node [vertices, label=right:{$5$}] (1) at (0,1){};
              \node [vertices, label=right:{$1$}] (2) at (0.75,0){};
              \node [vertices, label=right:{$3$}] (3) at (0.75,1){};
              \foreach \to/\from in {0/1, 1/2, 2/3}
      \draw [-] (\to)--(\from);

\end{tikzpicture}
\end{center}
Then $\overline{\Lambda}\phi$ is the monomial $x_a^2x_bx_c^3x_d^2$.
\end{example}

For antichains $A$ and $B$ in $P$ define an order relation by $A \leq B$ if
for every $b \in B$ there is an $a \in A$ with $a \leq b$
(but for $a \in A$ we do not require there to be $b \in B$ with $a \leq b$).
If $A$  and $B$ generate the filters $F_A$ and $F_B$, this corresponds
precisely to $F_A \supseteq F_B$.

Given an isotone map 
\begin{equation} \label{eq:regPN} 
\phi : P \pil \NN
\end{equation}
we get a filter $F_m = \phi^{-1}[m, +\infty \rangle$
and a descending chain 
\begin{equation} \label{eq:regChfilter}
P = F_0 \supseteq F_1 \supseteq F_2 \supseteq \cdots
\end{equation}
of poset filters. Let $A_i = \min F_i$. This is an antichain and the $A_i$
fulfill the order relation given above
\begin{equation} \label{eq:regAchfilter}
A_0 \leq A_1 \leq A_2 \leq \cdots .
\end{equation}

\begin{lemma} \label{lem:regFilters}
There are one-to-one correspondences between isotone maps \eqref{eq:regPN},
chains of poset filters \eqref{eq:regChfilter}, and chains of 
antichains \eqref{eq:regAchfilter}.
 Furthermore:
\begin{itemize}
\item[a.] $(p,i) \in \Lambda \phi$ iff $p \in A_{i+1}$
\item[b.] The projection
$\oLa \phi$ is $\sum_{i>0}  A_i$ considered as an element of $\NN P$.
\end{itemize}
\end{lemma}

\begin{proof}
The one-to-one correspondences are clear. That $(p,i) \in \Lambda \phi$
means $i < \phi(p)$ and $\phi(q) \leq i$ for every $q < p$. Then $p \in F_{i+1}$
and since $q$ is not in $F_{i+1}$ we get $p \in A_{i+1}$. 
Conversely if $p \in A_{i+1}$, then $\phi(p) \geq i+1$ and when $q < p$ 
then $q \not \in F_{i+1}$ and so $\phi(q) \leq i$. Thus $(p,i)$ in $\La \phi$.

For statement b. let $\phi(p) = r$ and $s = \max \{ \phi(q) \, | \, q < p \}$, 
so $\Lambda \phi$ contains $\{ (p,s),(p,s+1), \ldots, (p,r-1) \}$ and
these are all the elements with $p$ as first coordinate. Then 
$\oLa \phi = \cdots + (r-s)p + \cdots$. But $p$ is then in precisely
$A_{s+1}, \cdots, A_r$ and so $\sum_{i>0} A_i = \cdots + (r-s)p + \cdots$.  
\end{proof}

\begin{proposition} \label{pro:regLbar}
The map 
\[ \oLa : \Hom(P, \NN)  \pil \NN P \]
is a bijection.
\end{proposition}

\begin{proof}
Given an isotone $\phi : P \pil \NN$ we get by the above lemma
a filtration of poset filters 
\begin{equation} \label{eq:MonXPFilt}
P = F_0 \supseteq F_1 \supseteq \cdots \supseteq F_N = \emptyset. 
\end{equation}

\noindent{Injectivity of $\oLa$:} If $\phi \neq \psi$, then $F_i^\phi \neq F_i^\psi$
for some $i$. Let $i$ be minimal such. Then for say $\phi$ there is a minimal 
$p \in F^\phi_i$ such that
no $q \in F_i^\psi$ is $\leq p$.
Then $\oLa\phi = \cdots + j_1p + \cdots$ and 
$\oLa\psi = \cdots + j_2 p + \cdots$ where $j_2 < j_1$ (since $i$ is minimal)
and so $\oLa\phi \neq \oLa(\psi)$. 

\noindent{Surjectivity of $\oLa$:} Given $\bfa_1 = \sum a_{1p} p \in \NN P$. Let 
$A_1$ be the set of minimal elements of $\{p \, | \, a_{1p} \neq 0 \}$ 
and $F_1$ the filter generated by 
$A_1$. Considering the set $A_1$ as a formal sum in $\NN P$, 
let $\bfa_2 = \bfa_1 - A_1$ and let $F_2$ be the filter generated by $A_2$, 
the set of minimal elements $\min \{ p \, | \, a_{2p} \neq 0 \}$. 
Continuing we get a sequence
\[ P = F_0 \supseteq F_1 \supseteq \cdots \supseteq F_N = \emptyset.\]
This determines a map $\phi$ such that $\oLa\phi = \bfa_1$. 
\end{proof}


\begin{proposition}
If $\cJ$ is a poset ideal in $\Hom(P,\NN)$, then $\oLa\cJ$ is a poset
ideal in $\NN P$. 
\end{proposition}

\begin{proof}
Let $A_1 \leq A_2$ be antichains, and $D_1 \sus A_1$. Let $D_2$ be the largest
subset of $A_2$ such that $B_1 = (A_1 \backslash D_1) \cup D_2$ is an antichain.
(Note that there is a unique maximal such $D_2$.) 
Let $B_2 = A_2 \backslash D_2$.
Then $B_1 \leq B_2$ are antichains, $A_1 \leq B_1$ and $A_2 \leq B_2$, and
$A_1 + A_2 - D_1 = B_1 + B_2$ in $\NN P$. 
If we have a chain $A_1 \leq A_2 \leq A_3$ we may let 
$B_2 = (A_2 \backslash D_2) \cup D_3$ and $B_3 = A_3 \backslash D_3$. 
Again we have $A_3 \leq B_3$
and $A_1 + A_2 + A_3 - D_1$ equals $B_1 + B_2 + B_3$. 
In this way we may continue if we have longer chains.

Now given $\phi \in \cJ$. Let $n = \max_{p \in P} \phi(p)$. Then $\phi$
corresponds to a chain, with $A_i = \min F_i$ in \eqref{eq:MonXPFilt}:
\[ A_1 \leq A_2 \leq \cdots \leq A_n. \]
Furthermore $\oLa(\phi) = \sum_{i=1}^n A_i$. 
Let $p \in A_i$. We show that $\oLa\phi - p$ is also in $\oLa \cJ$, proving
that $\oLa \cJ$ is also a poset ideal.
If we remove a $p \in A_i$
from this, the above procedure gives a chain 
\[ B_1 \leq B_2 \leq \cdots \leq B_n \]
where $B_j = A_j$ for $j \leq i-1$ and $B_i = A_i \backslash \{ p \}$.  
Also $A_i \leq B_i$ for every $i$ and so this chain corresponds to 
an isotone map $\psi$ with $\psi \leq \phi$ and with 
\[ \oLa \psi = \sum_i  B_i = (\sum_i A_i) - p = \oLa \phi -p, \]
and so the latter is in $\oLa \cJ$.
\end{proof}

\begin{remark}
It is {\it not} true in general that 
if $Q \pil P$ is a bijective isotone map,
so $Q$ is a weakening of the partial order on $P$, the natural 
bijection 
\[ \oLa_Q^{-1} \circ \oLa_P : \Hom(P,\NN) \pil \Hom(Q,\NN) \]
takes poset ideals to poset ideals.
\end{remark}

Now we can round off with our goal of getting a description
of $L^p(\cJ,P)$ in $\kr[x_P]$ where
$p : P \times \NN \pil P$ is the projection map. 

\begin{corollary} \label{cor:regLbar}
The set of monomials in $L^p(\cJ,P)$ is precisely the
image by $\oLa$ of the filter $\cJ^c$. 
In other words, the normal (i.e. nonzero) monomials in $\kr[x_P]/L^p(\cJ,P)$ are
precisely the $\oLa\phi$ for $\phi \in \cJ$. 
\end{corollary}

\begin{proof}
$L(\cJ,P)$ is generated by $\Lambda\phi$ for $\phi \in \cJ^c$, so
$L^p(\cJ,P)$ is generated by the $\oLa\phi$
for $\phi \in \cJ^c$. But the image of $\cJ^c$ by $\oLa$ is a poset filter in 
$\NN P$ or equivalently a monomial ideal in $\kr[x_P]$. 
Hence the image is precisely $L^p(\cJ,P)$.
\end{proof}

\begin{example}
If $P$ is an antichain $\underline{m} = \{1,2,\ldots,m\}$, 
the map $\oLa$ of Proposition \ref{pro:regLbar}
is an isomorphism of posets. Hence in this case we get a one-to-one 
correspondence between poset filters in $\Hom(\underline{m},\NN)$ and 
monomial ideals in $\kr[x_{\underline{m}}]$, given by the map $\oLa$. The ascent
$\La$ sends the filter $\cJ^c$ to the letterplace ideal 
$L(\cJ,P)$ in $\kr[x_{P \times \NN}]$, which is a squarefree monomial ideal. 
This ideal is the {\it standard polarization}
of the monomial ideal $L^p(\cJ,P)$ in $\kr[x_P]$, whose normal
monomials are precisely $\oLa \cJ$.
\end{example}

In Section \ref{sec:ststable} we show that the image of $\oLa$ when
$P$ is a chain, is precisely the strongly stable ideals in $\kr[x_P]$.

\subsection{Which monomial ideals in $\kr[x_P]$ come from 
$\Hom(P,\NN)$?}

Let $I \sus \kr[x_P]$ be a monomial  ideal. We want to investigate when 
$I = \oLa\cJ$ for a poset ideal $\cJ \sus \Hom(P,\NN)$. First we need
some definitions. 

\begin{definition}
Given $b \in P$, a multichain 
\[ C : p_1 \leq p_2 \leq \cdots \leq p_r \] in $P$ is a {\it $b$-chain} if
$p_r \leq b$. The associated monomial $m_C$ is $\prod_{i=1}^r x_{p_i}$. 
The $b$-chain {\it goes through $a$} if $a \leq b$ and $a$ may be inserted
in $C$ to make it a larger multichain (so $a$ may or may not
be equal to one of the $p_i$'s).

The $b$-chain $C$ is {\it in a monomial} $m$ if $m_C$ divides $m$. 
The $b$-chain $C$ is a {\it longest} $b$-chain in $m$ if there is no
$b$-chain $C^{\prime}$ of $m$ with cardinality $|C^\prime| > |C|$.
Note that there may be several longest $b$-chains of $m$.
\end{definition}

\begin{lemma} If $m = \oLa \phi$ and $b \in P$, the longest $b$-chain
in $m$ has length $\phi(b)$. 
\end{lemma}

\begin{proof}
Suppose $(p,i)$ and $(q,j)$ are distinct elements in $\La \phi$
with $p \leq q \leq b$. 
If $p < q$ then $i < \phi(p) \leq j < \phi(q) \leq \phi(b)$.
If $p = q$ we may assume $i < j$ and we have $j < \phi(p) \leq \phi(b)$.
Thus in both cases $i < j < \phi(b)$. 
The consequence of this is that if 
$p_1 \leq p_2 \leq \cdots \leq p_r \,\,(\leq b)$ is a $b$-chain
in $m$, then $r \leq \phi(b)$.

\medskip
We now show that there is a $b$-chain in $m$ of length $r = \phi(b)$. 
If $b$ is minimal then 
\[ m = \oLa \phi = r \cdot b + \mbox{ other terms }, \]
and so clearly $b \leq b \leq \cdots \leq b$ repeated $r$ times is 
a longest $b$-chain in $m$. 

Suppose $b$ is not minimal and let $q < b$ be such that $s = \phi(q)$
is maximal. By induction on the height of elements of $P$ we may assume that
there is a $q$-chain $p_1 \leq \cdots \leq p_s$ in $m$. Since 
$(b,s), \ldots, (b,r-1)$ is in $\Lambda \phi$, this extends
to a $b$-chain
\[ p_1 \leq \cdots \leq p_s \leq b \leq \cdots \leq b \]
in $m$, with $b$ repeated $r-s$ times. This chain has length $r = \phi(b)$.
\end{proof} 

\begin{definition} An ideal $I \sus \kr[x_P]$ is {\it $P$-stable} if the
following holds: Let $m = nm_B$ be a monomial in $I$ where 
$B = \{ b_1, \ldots, b_r \}$ is an antichain in $P$. 
Let $a \in P$ and suppose for every $b \in B$ there is
a longest $b$-chain in $m$ going through $a$. Then $n \cdot x_a \in I$. 
\end{definition}

\begin{proposition} \label{pro:MXPPstabil} 
An ideal $I \sus \kr[x_P]$ is an image $I = \oLa\cJ^c$
for some poset filter $\cJ^c \sus \Hom(P, \NN)$ iff $I$ is $P$-stable.
\end{proposition}

We prove this after the remark and the example.

\begin{remark} If  $a < b$  and $x_b$ divides $m$, we cannot say
whether $\frac{mx_a}{x_b} \in I$. Only if there is a longest $b$-chain 
going through $a$ we can say this. Our notion  of $P$-stable is
therefore quite distinct from the notion of $P$-Borel in
\cite{FMS}.
\end{remark}

\begin{example}
Let $P$ be the poset
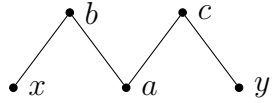
\begin{center}
\begin{tikzpicture}[scale=1, vertices/.style={draw, fill=black, circle, inner sep=1pt}]
              \node [vertices, label=right:{$x$}] (0) at (-1.5+0,0){};
              \node [vertices, label=right:{$a$}] (1) at (0,0){};
              \node [vertices, label=right:{$y$}] (2) at (1.5,0){};
              \node [vertices, label=right:{$b$}] (4) at (-.75,1){};
              \node [vertices, label=right:{$c$}] (5) at (.75,1){};
              \foreach \to/\from in {0/4, 1/4, 1/5, 2/5}
      \draw [-] (\to)--(\from);
\end{tikzpicture}
\end{center}
and $m = x^4a^2y^3 bc^2$ (we write $p$ instead of $x_p$ for the variables). 
The only longest $b$-chain in $m$ is $x \leq x \leq x \leq x \leq b$.
If $m$ is in a $P$-stable ideal $I$ then $\frac{mx}{b} \in I$. 
There is no longest $b$-chain through $a$, so even though $a < b$ we do not
need to have $\frac{ma}{b} \in I$.

\medskip
Let $P$ be 
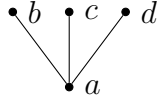
\begin{center}
\begin{tikzpicture}[scale=1, vertices/.style={draw, fill=black, circle, inner sep=1pt}]
              \node [vertices, label=right:{$a$}] (0) at (0,0){};
              \node [vertices, label=right:{$b$}] (1) at (-.75,1){};
              \node [vertices, label=right:{$c$}] (2) at (0,1){};
              \node [vertices, label=right:{$d$}] (3) at (.75,1){};
              \foreach \to/\from in {0/1, 0/2, 0/3}
      \draw [-] (\to)--(\from);
\end{tikzpicture}
\end{center}
and $m = ab^2c^3d^2$. Then each of $b,c,d$ have longest
chains in $m$ through $a$. Thus if $m$ is in a $P$-stable ideal $I$
then, letting $B = \{b,c,d\}$
\[ m^\prime = \frac{ma}{bcd}= a^2bc^2d \in I. \] Similarly 
\[ \frac{m^\prime a}{bcd} = a^3c \in I, \] 
\end{example}

\begin{remark}
When $P$ is a chain $[m] = \{1 < 2 < \cdots < m\}$, an ideal 
$I \sus \kr[x_{[m]}]$ is $P$-stable iff it is strongly stable,
see Section \ref{sec:ststable}. When $P$ is an antichain 
$\underline{m} = \{1,2, \ldots, m\}$ i.e. with no distinct comparables,
then any monomial  ideal  $I \sus \kr[x_{\underline{m}}]$ is $P$-stable.
\end{remark}

\begin{proof}[Proof of Proposition \ref{pro:MXPPstabil}.] We
will first show that if $I$ is $\oLa\cJ^c$ for some ideal $\cJ \sus
\Hom(P,\NN)$, then $I$ is $P$-stable. We divide into two parts. 
Recall that if $q > p$ and there are no $r$ with $q > r > p$, then 
$q$ is said to cover $p$. Let $m$ be a monomial in $I$. By Proposition
\ref{pro:regLbar} $m = \oLa\psi$ for some 
$\psi \in \Hom(P, \NN)$. Let $m = m_B \cdot n$ where $B$ is an 
antichain in $P$ and $a$ an element of $P$ such that for every $b \in B$
there is a longest $b$-chain through $a$. 

1. In the first part, let each of the elements of the set $B$ cover $a$.
For $b \in B$ a longest $b$-chain in $m$ through $a$, is then longer than a
a longest $a$-chain in $m$, and so $\psi(b) > \psi(a)$. 
Let $C \supseteq B$ be the set of all covers of $a$. 

Let $C^\prime \sus C$ be the $c \in C$ such that $\psi(c) = \psi(a)$.
Then a longest $c$-chain and a longest $a$-chain in $m$ have the same
length. Also $C^\prime$ is disjoint from $B$. Let $B^\prime = B \cup C^\prime$,
let $m^\prime = m \cdot m_{C^\prime}$ and let $m^\prime$ correspond to 
$\psi^\prime$ by $\oLa$. For $c \in C^\prime$ the longest $c$-chain in $m^\prime$
is one more than the longest $a$-chain in $m^\prime$. So $\psi^\prime(c) = 
\psi(a) + 1$. For $c \in (C \backslash C^\prime) \cup B$ clearly 
$\psi (c) > \psi(a)$ and so $\psi^\prime (c) > \psi(a) = \psi^\prime(a)$.
Thus $\psi^\prime(c) > \psi^\prime(a)$ for every $c$ covering $a$. 
Define $\psi^{\prime \prime}$ by 
\[ \psi^{\prime \prime}(p) = \begin{cases} \psi^\prime(p), & p \neq a \\
\psi^\prime (a) +1, & p = a \end{cases}.  \] 
Thus $\psi^{\prime \prime}$ is an isotone map covering $\psi^\prime$,
and $\psi^{\prime}$ corresponds to $m^\prime = n \cdot m_B \cdot m_{C^\prime}$
and $\psi^{\prime \prime}$ corresponds to $m^{\prime \prime} = n x_a$.

\medskip
2. In the second part we now induct on the longest chain from any
$b \in B$ to $a$. Let $a_1, \ldots, a_r$ be the elements covering $a$.
Then for every $b \in B$ there is a longest $b$-chain through some $a_i$
and $a$. Partition  $B = B_1 \cup \cdots \cup B_r$ such that the elements in 
$B_i$ have a longest $b$-chain through $a_i$ and $a$. By induction 
$n x_{a_1} m_{B_2} \cdots m_{B_r} \in I$ and $n x_{a_1}x_{a_2}m_{B_3} \cdots
m_{B_r} \in I$ and in the end $n x_{a_1} \cdots x_{a_r} \in I$. 
By part 1. above we now get $n x_a \in I$. 

\medskip
We now show that if $I$ is $P$-stable, then $I = \oLa\cJ^c$ for
some poset ideal $\cJ \sus \Hom(P, \NN)$. By Proposition \ref{pro:regLbar}, 
$I = \oLa\cF$ for some subset $\cF \sus \Hom(P, \NN)$. 
We want to show that $\cF$ is a filter. Given $\psi \in \cF$ and
$\tilde{\psi} > \psi$. Choose a maximal $a$ such that 
$\tilde{\psi}(a) > \psi(a)$.
Define $\psi^\prime$ by 
\[ \psi^\prime(p) = \begin{cases} \psi(p), & p \neq a \\
                       \psi(a) +1, & p = a
                    \end{cases}
\]
Then $\psi^\prime$ is an isotone map. Let $C \sus P$ be the elements 
$c$ covering $a$,
so $\psi(a) < \psi(c)$ for $c \in C$ (by definition of $a$). 
Let $C^\prime \sus C$ be the subset with 
a longest chain in $m$ through $a$. If $m = \oLa\psi$ then 
$\oLa\psi^\prime = \frac{mx_a}{m_{C^\prime}} $ which is in $I$
since $I$ is $P$-stable,
and so $\psi^\prime \in \cF$. In this way we can successively increase 
$\psi$ by isotone maps in $\cF$ and eventually reach $\tilde{\psi}$. 
Thus the latter is in $\cF$ and so $\cF$ is a poset filter.
\end{proof}

\begin{example}
The ideal $I = (a,b,c)^2 \sus \kr[a,b,c]$ is $P$-stable when $P$ is either the
chain or the antichain:
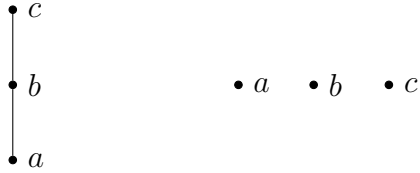
\begin{center}
\begin{tikzpicture}[scale=1, vertices/.style={draw, fill=black, circle, inner sep=1pt}]
              \node [vertices, label=right:{$a$}] (0) at (-2,-1){};
              \node [vertices, label=right:{$b$}] (1) at (-2,0){};
              \node [vertices, label=right:{$c$}] (2) at (-2,1){};
              \node [vertices, label=right:{$a$}] (4) at (1,0){};
              \node [vertices, label=right:{$b$}] (5) at (2,0){};
              \node [vertices, label=right:{$c$}] (6) at (3,0){};
              \foreach \to/\from in {0/1, 1/2}
      \draw [-] (\to)--(\from);
\end{tikzpicture}
\end{center}
However it is not $P$-stable if $P$ is
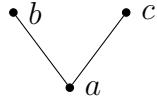
\begin{center}
\begin{tikzpicture}[scale=1, vertices/.style={draw, fill=black, circle, inner sep=1pt}]
              \node [vertices, label=right:{$a$}] (0) at (0,0){};
              \node [vertices, label=right:{$b$}] (1) at (-.75,1){};
              \node [vertices, label=right:{$c$}] (2) at (.75,1){};
              \foreach \to/\from in {0/1, 0/2}
      \draw [-] (\to)--(\from);
\end{tikzpicture}
\end{center}
Since $bc \in I$, being $P$-stable would imply $a \in I$, which is not so.
\end{example}

\begin{proposition} Let $\mm = (x_p)_{p \in P}$ be the maximal ideal of 
$\kr[x_P]$, and $d \geq 2$. Then the power $\mm^d$ is $P$-stable iff $P$
is a disjoint union of posets whose Hasse diagrams are rooted trees, with the
roots on top. 
\end{proposition}

\begin{proof}
Suppose there are elements $a, b_1, b_2$ of $P$ with $a < b_1$ and 
$a < b_2$
and $\{ b_1, b_2 \}$ an antichain. Then $m = b_1b_2^{d-1} \in \mm^d$. 
The longest $b_1$-chain in $m$ is $b_1$ which goes through $a$. 
The longest $b_2$-chain is $b_2 \geq \cdots \geq b_2$ a number of $d-1$ times,
which also goes through $a$. By Proposition \ref{pro:MXPPstabil},
$ab_2^{d-2} \in \mm^d$ which is not so. Thus each $a$ in $P$ has at most
one cover and so $P$ is a disjoint union of rooted trees with roots
at the top.

On the other hand let $P$ be a disjoint union of posets whose Hasse
diagrams are trees with roots on top. If $B$ is an antichain consisting
of elements $\geq a$, then $B = \{ b \}$ for a single element $b$.
If a longest $b$-chain in a monomial goes through $a$, the criterion
for a monomial ideal $I \sus \kr[x_P]$ being $P$-stable is that if 
$x_b m \in I$ then $x_a m \in I$.  This holds for $I = \mm^d$.
\end{proof}

\subsection{Primary decomposition}
Let $S$ be a subset of $P$. A monomial prime ideal in $\kr[x_P]$ is
of the form $\fp(S) = (x_p)_{p \in S}$. 

\begin{lemma}
Let $I \sus P$ be a poset ideal. Then $\fp(I)$ is a $P$-stable monomial
prime ideal. All $P$-stable monomial prime ideals are of this form.
\end{lemma}

\begin{proof}
If $I \sus P$ is a poset ideal, then clearly $\fp(I)$ is $P$-stable. 
Conversely, let $S \sus P$ and suppose $\fp(S)$ is $P$-stable. Let $b \in S$
so $x_b \in \fp(S)$. Then for $a < b$, the longest $b$-chain in $x_b$
goes through $a$ (it is just $b$ itself), and so $x_a \in \fp(S)$. 
Thus $S$ is a poset ideal.
\end{proof}

\begin{proposition}
Let $L \sus \kr[x_P]$ be a $P$-stable ideal. If $\fp$ is an associated
prime ideal of $L$, then $\fp = \fp(I)$ for some poset ideal $I \sus P$.
\end{proposition}

\begin{proof}
Let $\fp(S)$ be a prime ideal annihilating $\overline{m} \in \kr[x_P]/L$,
so $m x_b \in L$ for $b \in S$. Let $a < b$. Then $mx_a^r x_b \in L$.
For $r$ large the longest $b$-chain will go through $a$. Therefore
$mx_a^{r+1} \in L$. If $mx_a \not \in L$ then $x_a \not \in \fp(S)$ but
$x_a^{r+1} \in \fp(S)$. This contradicts $\fp(S)$ being prime.
\end{proof}

\section{Principal poset ideals} \label{sec:fin}

In this section $\cJ \sus \Hom(P, \NN)$ will be a finite poset ideal.
Then we actually have $\cJ \sus \Hom(P, [n])$ for some $n$ and 
this is the situation studied in the previous articles
\cite{FGH} 
 and \cite{DFN-COLP}. 

By Corollary \ref{cor:regLbar} the monomial  ideal
$L^p(\cJ,P) \sus \kr[x_P]$ is then an artinian monomial ideal,
and so a Cohen-Macaulay ideal.
Since this quotient ring is obtained by cutting down from 
$\kr[x_{\Supp(\cJ)}]/L(\cJ,P)$ by a regular sequence, we see,
\cite[Thm. 5.9]{FGH}, that $L(\cJ,P)$ is a Cohen-Macaulay ideal
when $\cJ$ is a finite poset ideal.

\begin{remark} In \cite{DFN-COLP} it is shown that when $\cJ$ is a
finite poset, the letterplace ideal $L(\cJ,P)$ defines a simplicial
ball (save a few exceptions). Its boundary is therefore a simplicial 
sphere, whose Stanley-Reisner ideal is also precisely described
in \cite[Section 5]{DFN-COLP}
\end{remark}

For a finite poset ideal $\cJ$ of $\Hom(P, \NN)$ we define the 
{\it hull map} $\alpha(p) = \max \{ \phi(p) \, | \, p \in \cJ \}$.

\begin{lemma}
When $\cJ \sus P \times \NN$ is a finite poset ideal, its hull
map $\alpha$ is an isotone map, and
\[ \Supp(\cJ) = \{ (p,i) \, | \, i \leq \alpha(p) \} \]
This support is a finite poset ideal in $P^{\op} \times \NN$. 
\end{lemma}

\begin{proof}
This is immediate to verify.
\end{proof}

\begin{remark}
When $\cJ$ is not a finite poset ideal, the hull map needs
not be isotone, and the support $\Supp(\cJ)$ needs not be a poset
ideal in $P^\op \times \NN$.
\end{remark}

A class of finite poset ideals now comes out as distinguished: The
poset ideals consisting av {\it all} isotone maps $\phi$ less than or equal to
the hull map.

\begin{definition}
Given an isotone map $\alpha : P \pil \NN$ it induces a finite
poset ideal in $\Hom(P,\NN)$:
\[ \cJ(\alpha) = \{ \psi  \, | \, \psi \leq \alpha \}. \]
The {\it principal} co-letterplace ideal $L(P,\alpha)$ is $L(P,\cJ(\alpha))$ and
the {\it principal} letterplace ideal $L(\alpha,P)$ is its Alexander dual 
$L(\cJ(\alpha),P)$. 
\end{definition}

\begin{example}
When $P$ is an antichain $\underline{m} = \{1,2,\cdots,m \}$, 
and $p : \underline{m} \times \NN \pil \underline{m}$ the projection
to the first coordinate, then $L^p(\alpha,P)$ is the complete intersection
of monomials $x_1^{\alpha(1)+1}, x_2^{\alpha(2)+1}, \ldots, x_m^{\alpha(m)+1}$. 
Thus principal letterplace ideals for general $P$ 
may in some way be considered as
generalizations of complete intersections.
\end{example}

\begin{example} When $\alpha$ is the constant function $\alpha(p) = n+1$, 
then the principal co-letterplace ideal $L(P,\alpha)$ 
is the co-letterplace ideal 
$L(P,n+1)$ of \cite{FGH} (note that we get $n+1$ here since in 
\cite{FGH} the chain $[n]$ starts with $1$ but here $\mathbb{N}$ starts
with $0$, so the chain $[n+1]$ is isomorphic to the chain $[0,n]$), 
and the Alexander dual $L(\alpha,P)$ is the
letterplace ideal $L(n+1,P)$, see Proposition \ref{pro:finPaiso} below.
In \cite{FN-Rigid} the author and
A. Nematbakhsh computes the full deformation families of the letterplace
ideals $L(2,P)$ when the Hasse diagram of $P$ is a rooted tree. The
deformation theory, in analog with complete intersections, is extremely nice.
\end{example}

In the case of a principal poset ideal $\cJ(\alpha)$ 
there is a quite nice description of the minimal
generators of the letterplace ideal $L(\alpha,P)$.

For an isotone map $\alpha : P \pil \NN$
let $[0,i]$ be the integers from $0$ to $i$, and
define the poset ideal $I_i = \alpha^{-1}([0,i])$.
We get a filtration of poset ideals of $P$:
\[ I_0 \sus I_1 \sus \cdots \sus I_n \sus \cdots  .\]

\begin{theorem} \label{pro:finPaiso}
The principal letterplace ideal 
$L(\alpha,P)$ is generated by the monomials $m_{\Gamma \phi}$
associated to isotone maps $\phi : [0,i] \pil P$
such that $\phi(i) \in I_i$. The minimal generators are such
$\phi$ with $\phi(j) \not \in I_j$ for $j < i$. 
\end{theorem}


The notion of principal ideal in a ring is an ideal generated by one element.
In contrast, principal letterplace ideals may have many generators.

\begin{example}
Let $\alpha : [3] \pil \NN$ be the isotone map
\[ 1 \mapsto 1, \quad 2 \mapsto 1, \quad 3 \mapsto 2. \]
The variables involved in the hull of $\alpha$ are the variables in
\begin{equation*}
\left [ 
\begin{matrix}  &  & x_{32} \\
x_{11} & x_{21} & x_{31} \\
x_{10} & x_{20} & x_{30} 
\end{matrix} \right ], 
\qquad
\begin{matrix}  2 \\ 1 \\ 0 
\end{matrix}
\underset{\begin{matrix} 1 & 2 & 3 \end{matrix}}
{\left [ \begin{matrix}  &  & \po \\
\po & \po & \po \\
\po & \po & \po 
\end{matrix} \right ]. }
\end{equation*}
The letterplace ideal $L(\alpha, [3])$ is minimally generated 
by the following monomials
\[ \begin{matrix} x_{10}x_{11} \\
x_{10}x_{21} \\
x_{20}x_{21}
\end{matrix}, \quad 
\begin{matrix} x_{10}x_{31}x_{32} \\
x_{20}x_{31}x_{32} \\
x_{30}x_{31}x_{32}.
\end{matrix}
\]
In general for an isotone map $\alpha : [m] \pil \NN$
the minimal generators of $L(\alpha,[m])$ are 
\[x_{\phi(0),0}x_{\phi(1),1}\cdots x_{\phi(r),r}\]
with $\phi$ isotone and $x_{\phi(j),j}$ in column $\phi(j)$ and row $j$. The last
variable $x_{\phi(r),r}$ is in top of its column i.e. $\alpha \circ \phi(r) = r$, 
while the previous variables are not in top of their columns, i.e.
$\alpha \circ \phi(j) > j$.
\end{example}

Before proving the above theorem we develop
some lemmata.

\begin{lemma} \label{lem:DefiStepdown}
Let $(p,i) \in \Lambda \phi$ where $i \geq 1$. Then there is a 
$q \leq p$ such that $(q,i-1) \in \Lambda\phi$.
\end{lemma}

\begin{proof}
If every $q < p$ has $\phi(q) < i$ then $(p,i-1) \in \Lambda\phi$.
Otherwise there is a minimal $q < p$ with $\phi(q) = i$. 
Then $(q,i-1) \in \Lambda\phi$. \end{proof}

\begin{lemma} \label{lem:DefiTopp} 
Suppose $\psi$ is not $\leq \phi$. Then there is $(p,i)
\in \Lambda\psi$ such that $i \geq \phi(p)$. 
\end{lemma}

\begin{proof}
Let $p$ be minimal such that $\psi(p) > \phi(p)$. 
Let $i = \psi(p)- 1 \geq \phi(p)$. Then if $q < p$  
\[ \psi(q) \leq \phi(q) \leq \phi(p) <  \psi(p). \] 
So we see that $(p,i) \in \Lambda \psi$.
\end{proof}

\begin{lemma} \label{lem:DefiMultichain} 
Suppose $\phi(p) = r$. For any multichain $p_0 \leq p_1 \leq
\cdots \leq p_r = p$, there is some $\psi$ which is not $\leq \phi$
such that 
\[ \Lambda\psi = \{ (p_0,0), (p_1,1), \ldots, (p_r,r) \}. \]
\end{lemma}

\begin{proof}
Let $F_0 = P$ and for $i = 1, \ldots, r+1$ let $F_{i}$ consist of all elements 
$\geq p_{i-1}$. For $p \in F_i \backslash F_{i+1}$ define $\psi(p) = i$.
Note that if 
\[ p_{u-1} < p_u = \ldots = p_v < p_{v+1},\] 
then $\psi(p_u) = \cdots = \psi(p_v) = v+1$ and $\psi(p_{u-1}) = u$. 
We see that $\psi(p_r) > r$ and so $\psi$ is not $\leq \phi$.
We see that all $(p_i,i)$ are in $\Lambda \psi$, and will show that
these are precisely the elements of $\Lambda \psi$.

Let $(q,t) \in \Lambda \psi$
with $\psi(q) = i+1$. Then $t\leq i$ and $q \in F_{i+1} \backslash F_{i+2}$
so $q \geq p_i$. Then since $F_{i+1} \backslash F_{i+2}$ is nonempty,
$p_{i+1} > p_i$ and so $\phi(p_i) = i+1$. If $q > p_i$ we could not
have $(q,t) \in \Lambda\psi$. Thus $q = p_i$.
Let  $p_{j-1} < p_j = \cdots = p_i < p_{i+1}$. Then $\phi(p_{j-1}) = j$
and we must have $j \leq t \leq i$. Thus 
$(q,t) = (p_t,t)$.
\end{proof}

\begin{proposition} Given a principal poset ideal $\cJ(\phi)$.
The inclusion minimal elements of the sets 
$\Lambda\psi$ for $\psi \in  \cJ(\phi)^c$ are sets
\[ \{ (p_0,0), (p_1,1), \ldots, (p_r,r) \} \] 
where:
\begin{itemize}\item $p_0 \leq p_1 \leq \cdots \leq p_r$ is a multichain
\item $\phi(p_r) = r$ 
\item $\phi(p_i) > i$ for $i < r$. 
\end{itemize}
\end{proposition}

\begin{proof} By Lemma \ref{lem:DefiMultichain} every such chain is 
$\Lambda \psi$ for some $\psi$ not  $\leq \phi$. 
By Lemma \ref{lem:DefiTopp} we cannot omit $(p_r,r)$ from such a chain, and
by Lemma \ref{lem:DefiStepdown} we cannot omit $(p_i,i)$ for $i < r$.
\end{proof}

\begin{proof}[Proof of Theorem \ref{pro:finPaiso}.]
This is immediate from Definition \ref{def:defiLP} and the proposition above.
\end{proof}


\begin{example}
A class of determinantal ideals considered by Bruns and Vetter in \cite{BrVe} 
and Herzog and Trung in \cite{HeTr}
comes from a pair of integer sequences
$M$:
\begin{align*}
0 &  = a_0 <  a_1 < a_2 < \cdots < a_r < a_{r+1} = m+1 \\
0 &  = b_0 <  b_1 < b_2 < \cdots < b_r < b_{r+1} = n+1.
\end{align*}
The ideals $I_M$ is generated by the $t$-minors of the matrix $(x_{ij})$ whose
indices is the hook subset of $[1,m] \times [1,n]$ consisting of pairs $(a,b)$
where $a < a_t$ or $b < b_t$, and successively letting $t = 1,2, \ldots, r+1$.
We explain here how the initial ideal of 
$I_M$ for a diagonal term order is a regular quotient of certain
principal letterplace ideals.

\medskip
Given sequences
\begin{align} \label{eq:FinPosAlpha}
0 = \alpha_0 \leq & \alpha_1 \leq \alpha_2 \leq \cdots \leq \alpha_r 
\leq \alpha_{r+1} = m-r \\
0 = \beta_0 \leq & \beta_1 \leq \beta_2 \leq \cdots \leq \beta_r \leq 
\beta_{r+1} = n-r, 
\notag
\end{align}
for $i = 0, \ldots, r$ define $H_i$ to be the
hook subset of $[\alpha_i+1,m-i] \times [\beta_i + 1,n-i]$ consisting of all
elements $(a,b)$ with $a \leq \alpha_{i+1}$ or $b \leq \beta_{i+1}$. 
(So $H_0$ is the whole $[1,m] \times [1,n]$ matrix with the 
submatrix $[\alpha_1 + 1,m] \times [\beta_1 +1,n]$ removed.)
Let $C_i = H_0 \cup H_1 \cup \cdots \cup H_i$ and 
$C = \cup_{i = 0}^r H_i$ which is a subset of $[1,m] \times [1,n]$ with
the induced poset structure. The filtration 
$C_0 \sus C_1 \sus \cdots$ defines a map $\tau : C \pil \NN$.
The support of the principal letterplace ideal $L(\tau,C)$ is the
following subset of $C \times \NN$:
\[ \tilde{C} = (C \times \{0 \}) \cup ((C \backslash C_0) \times \{ 1 \})
\cup ((C \backslash C_1) \times \{ 2 \}) \cup \cdots .\]
This is a finite subset. 
Let $\phi: \tilde{C} \pil \NN \times \NN$ be the natural
map sending an element $(a,b,p)$ of $(C \backslash C_{p-1}) \times \{p \}$ to 
$(a+p,b+p)$. The map $\phi$ is an isotone map with bistrict chain fibers.
The image of $\phi$ is in $B = [1,m] \times [1,n]$ and we get the ideal
$L^\phi(\tau,C) \sus \kr[x_B]$. 
For the sequences \eqref{eq:FinPosAlpha} we let $a_i = \alpha_i + i$
and $b_i = \beta_i + i$. Then the initial ideal of $I_M$ with respect
to a diagonal term order, is the ideal $L^\phi(\tau,C)$.

\end{example}

\section{Determinantal ideals specializing to principal 
letterplace ideals} \label{sec:detmax}

We consider isotone maps $\phi : [m] \pil \NN$ where $[m]$ is the chain
on $m$ elements. We will find
determinantal ideals whose initial ideals are the principal letterplace
ideals $L(\phi,[m])$. 
This is a class of determinantal ideals which seemingly has not 
been considered before. It is a class of determinantal ideals
which naturally generalizes the ideals generated by the {\it maximal
minors} of a matrix of distinct variables: If $\phi : [m] \pil \NN$ is
the constant map sending each $i \mapsto n$, then $L(\phi, [m])$ is the
initial ideal of the ideal of maximal minors of an $(n+1) \times (m+n)$ matrix
of distinct variables.
  
The naturality of this wider class of ideals is shown by the following
analogy:  If $\underline{m}$ is the antichain on $m$ elements, 
and $\psi : \underline{m} \pil \NN$ the constant map sending $i \mapsto n$,
the letterplace ideal $L(\psi, \underline{m})$ is the complete intersection
of $m$ monomials all of the same degree $n$. On the other hand if $\psi$ sends
$i \mapsto d_i$, then the letterplace ideal is the complete intersection of
monomials of degrees $d_1+1, d_2+1, \ldots, d_m+1$.

\subsection{A new class of determinantal ideals}
\label{subsec:DetmaxDefi}
Given a sequence 
\begin{equation} \label{eq:Detl}
\ell : \ell_a \leq \ell_{a+1} \leq \cdots \leq \ell_b. 
\end{equation}
To this sequence we associate a matrix $M(\ell)$ with columns
indexed by $\ell_a + 1, \ell_a + 2, \ldots, \ell_b$ and 
rows indexed by $a, a+1, \ldots, b-1$. 
  For $p \in [\ell_c + 1, \ell_{c+1}]$ we put the variable $y_{p,i}$ in position
$(p,i)$ if $a \leq i \leq c$ and we put $0$ in position $(p,i)$ if
$i > c$. Let $k[Y;\ell]$ be the polynomial ring generated by 
these variables.

For $a < c \leq b$ denote by $J_a^c$ the ideal  generated by the
$(c-a)$-minors (maximal minors) of the submatrix of $M(\ell)$ whose columns are 
$\ell_a + 1, \ell_a + 2, \ldots, \ell_c$ and rows $a,a+1, \ldots, c-1$.
Let 
\[ I(\ell) = \sum_{c = a+1}^b J_a^c. \]

\begin{example} \label{ex:DetRun}
To the sequence
 \begin{tabular}{c|c|c|c|c}
$l_0$ & $l_1$ & $l_2$ & $l_3$ & $l_4$ \\
\hline 
$0$ & $0$ & $3$ & $4$ & $6$
\end{tabular}
there is associated a $4 \times 6$ matrix with variables in the 
indicated positions
\[ \left [ \begin{matrix}  & & & & \po & \po \\
                           & & & \po & \po & \po \\
                           \po &\po &\po &\po &\po &\po \\
                          \po &\po &\po &\po &\po &\po
\end{matrix} \right ] \]
The ideal $I(\ell)$ is the ideal generated by the $2$-minors of the lower
left $2 \times 3$ matrix, the $3$-minors of the lower left $3 \times 4$-matrix,
and the $4$-minors of the whole matrix.
\end{example}

\begin{example} \label{ex:DetIndstart}
Let 
\[ \ell : 0 = \ell_0 = \cdots = \ell_{n} \leq \ell_{n+1} = n+m. \]
Then $M(\ell)$ is an $(n+1) \times (n+m)$-matrix of distinct variables.
The ideal $I(\ell)$ is the ideal of maximal minors of $M(\ell)$.
It is shown in \cite[Subsec.3.3]{FGH}, or originally \cite{Stu}, that the 
initial ideal of $I(\ell)$ is the letterplace ideal $L(n+1,[m])$ (with the
variables suitably reindexed). This example is our induction start for
the proof of the main Theorem \ref{thm:DetMain} of this section
\end{example}

We shall show that $I(\ell)$ is a Cohen-Macaulay ideal of codimension
the maximum of $\ell_d - \ell_a - (d-a) + 1$ for $d = a+1,\ldots, b$,
and that its initial ideal is a principal letterplace ideal $L(\phi,[m])$. 
It is worth noting the following.

\begin{lemma}
Let  $a < c \leq b$ and suppose $l_d - l_c \leq d-c$ for $d \geq c$.
Then $I(\ell_a, \ldots, \ell_b) = I(\ell_a, \ldots, \ell_c)$.
\end{lemma}

\begin{proof}
Let $c < d \leq b$ and consider a determinant $D$ of size $d-a$ in $J_a^d$.
It is given by columns in $M(\ell)$ in positions $p_{a+1}, p_{a+2}, \ldots,
p_d$, and $p_d \leq \ell_d$. Then by hypothesis $p_c \leq \ell_c$. 
If $p_{c+1} \leq \ell_c$, then the upper $d-c$ rows has at most
$d-c-1$ nonzero columns, and the determinant is zero. If
$p_{c+1} > \ell_c$ then by the hypothesis we must have $p_r = \ell_r = \ell_c
+ r-c$, for $r = c+1, \ldots, d$. The upper $d-c$ rows then is only nonzero
in the last $d-c$ columns, which is a lower triangular matrix with
determinant $D_2$. 
Thus the determinant is $D$ is a product of $D_1 \cdot D_2$ where
$D_1$ is a $(c-a)$-determinant consisting of the first $c-a$ columns and
the rows $p_{a+1}, \ldots, p_c$.
\end{proof}

\subsection{Sequences and letterplace ideals}
For integers $s \leq t$ let $[s,t]$ be the interval of integers
$r$ such that $s \leq r \leq t$. From an isotone map $\phi:[1,m] \pil \NN$
we get the letterplace ideal $L(\phi,[m])$. To gain more flexibility
we can take an interval $[u+1,v]$ with $m = v-u$ and an integer $a \geq 0$ 
and identify 
$\phi$ with a map $[u+1,v] \pil \NN_{\geq a}$ and get a letterplace ideal 
$L(\phi,[u+1,v])$ which we can identify with $L(\phi,[m])$. 
The isotone map $\phi : [u+1,v] \pil \NN_{\geq a}$ may be given by a sequence
\begin{equation} \label{eq:Deti} i : u = i_a \leq i_{a+1} \leq \cdots 
\leq i_b = v,
\end{equation}
such that $\phi(p) = c$ for $p \in [i_c+1,i_{c+1}]$. 
Let $\kr[X;i]$ be
the polynomial ring with variables $x_{p,i}$ where $a \leq i \leq c$ for 
$p \in [i_c +1, i_{c+1}]$. We get a principal letterplace ideal
$L(\phi;[i_a +1, i_b])$ in $\kr[X;i]$, which we denote by $L(i)$. 

To the sequence $i$ we associate a sequence $\ell$ given by 
\[ \begin{cases} l_a = i_a+a, & \\
l_c = i_c + c-1, & i_c > i_{c-1} \\
l_c = l_{c-1}, & i_c = i_{c-1}.
\end{cases} \]
Let $1 \leq c < d \leq b$ be indices such that $l_d > l_{d-1}$ and 
$l_c > l_{c-1}$, so $l_d$ and $l_c$ are the last
columns in $M(\ell)$ with precisely $d$, resp. $c$ variables, then
\[ l_d - l_c = i_d - i_c + (d-c) > d-c. \]
A sequence with this property is called a {\it terrace sequence}.
It means that in the associated matrix $M(\ell)$, for a step the horizontal
part is $>$ the vertical part, except possibly for the first step where
one has $\geq$.  
Given a terrace sequence, we may associate a weakly increasing sequence
\eqref{eq:Deti} by 
\[ \begin{cases} i_a = l_a - a \\ 
i_c = l_c - c+1, & l_c > l_{c-1} \\
i_c = i_{c-1}, & l_c = l_{c-1}.
\end{cases} \]

This gives a one-to-one correspondence between weakly increasing sequences
$i$ and terrace sequences $\ell$.
To facilitate our argument later, we need to introduce some more 
notions on sequences.
Given any weakly increasing sequence $\ell$ we can associate a
terrace sequence $\ell^\prime$ as follows. Consider 
\begin{equation} \label{eq:DetmaxMax}
l_a - a, l_{a+1} -a, l_{a+2} - (a+1), \cdots, l_b - (b-1). 
\end{equation}
Let the successive maxima after $m_0 = l_a - a$ 
as we move from left to right be for indices $m_1, \ldots, m_r$
where by convention the last maximum is $\infty$ in position $m_r = b+1$. 
\begin{example}
Let 
\[ l_2 = 3, 3, 5, 7, 8, 11 = l_7. \]
The associated matrix $M(\ell)$ of $y$-variables is (with bullets indicating variables):
\begin{equation*}
\begin{matrix} 6 \\ 5 \\ 4 \\ 3 \\ 2
\end{matrix} 
\underset{ \begin{matrix} 4 & 5 & 6 & 7 & 8 & 9 & 10 & 11 \end{matrix} }
{\left [ 
\begin{matrix}
  &   &  &  &  & \po & \po & \po \\
  &   &  &  & \po & \po & \po & \po \\
  &   & \po & \po & \po & \po & \po & \po \\
\po & \po & \po & \po & \po & \po & \po & \po \\
\po & \po & \po & \po & \po & \po & \po & \po \\
\end{matrix}
\right ]}
\end{equation*}

The difference sequence is
\[ l_2 - 2 = 1, 1, 2,3,3,5 = l_7 - 6. \]
The successive maxima after $1$ as we move from left to right are $2,3,5$ 
and $\infty$ 
in positions $4,5,7$ and $8$ (by the convention).
\end{example}

Let $\ell^\prime_{m_i} = \ell_{m_i}$ and for $p \in [m_i, \ldots m_{i+1}-1]$ 
let $\ell^\prime_p = \ell_{m_i}$. This is a terrace sequence, actually the
largest terrace sequence which is $\leq  \ell$. And this terrace sequence
corresponds to a weakly increasing sequence $i$. 
Note that the matrix $M(\ell)$ is a submatrix of $M(\ell^\prime)$. 

\begin{example}
Continuing the example above we get
\[ l_2^\prime = 3, 3, 5, 7, 7, 11 = l^\prime_7. \]
The associated matrix $M(\ell^\prime)$ of $y$-variables is (with bullets indicating variables):
\begin{equation*}
\begin{matrix} 6 \\ 5 \\ 4 \\ 3 \\ 2
\end{matrix} 
\underset{ \begin{matrix} 4 & 5 & 6 & 7 & 8 & 9 & 10 & 11 \end{matrix} }
{\left [ 
\begin{matrix}
  &   &  &  & \po & \po & \po & \po \\
  &   &  &  & \po & \po & \po & \po \\
  &   & \po & \po & \po & \po & \po & \po \\
\po & \po & \po & \po & \po & \po & \po & \po \\
\po & \po & \po & \po & \po & \po & \po & \po \\
\end{matrix}
\right ]}
\end{equation*}
The associated $i$-sequence to $\ell$ and $\ell^\prime$ is 
\[ i_2 = 1, 1, 2, 3, 3, 5 = i_7. \] We may display the associated set of
$x$-variables as:
\begin{equation*}
\begin{matrix} 6 \\ 5 \\ 4 \\ 3 \\ 2
\end{matrix} 
\underset{ \begin{matrix} 2 & 3 & 4 & 5  \end{matrix} }
{\left [ 
\begin{matrix}
  &   & \po & \po \\
  &   & \po & \po \\
  & \po & \po &  \po \\
\po & \po & \po & \po \\
\po & \po & \po & \po \\
\end{matrix}
\right ]}
\end{equation*}
We may think of the columns of the matrix of $x$-variables above as 
being transported
to diagonals in the matrices $M(\ell)$ and $M(\ell^\prime)$, 
but we see that there will be
additional $y$-variables in these matrices which we use for determinants.
\end{example}

\subsection{The main statement}
We get a map of polynomial rings $\kr[X;i] \pil \kr[Y;\ell]$ given by 
$x_{p,i} \mapsto y_{p+i,i}$ so we are shifting the columns of the $x$-matrix
to north-east diagonals in the $y$-matrix. The image of the letterplace ideal $L(i)$ then
generates an ideal in $\kr[Y;\ell]$ which we denote $L^Y(i)$.

\begin{theorem} \label{thm:DetMain}
Given a weakly increasing sequence $\ell$, let $\ell^\prime$ the largest
terrace sequence $\leq \ell$, and $i$ its associated weakly increasing
sequence. For a diagonal term order on $\kr[Y;\ell]$ we have 
$\text{in} (I(\ell)) = L^Y(i)$. 

Hence $I(\ell)$ is a Cohen-Macaulay ideal of codimension equal to 
\[ i_b - i_a = l^\prime_b - l^\prime_a - (b-a) +1 
= \max \{l_d - l_a - (d-a) + 1\,|\, d = a+1, \ldots, b\}. \]
\end{theorem}

Before proving this we need some lemmata.
If $i^\prime$ is a segment of the sequence $i$, then $\kr[X;i^\prime]$
is a subring of $\kr[X;i]$. The ideal generated by $L(i^\prime)$ in 
the latter ring is denoted $\hat{L}(i^\prime)$. Similarly if $\ell^\prime$
is a segment of $\ell$, then $\kr[Y;\ell^\prime]$ is a subring
of $\kr[Y,\ell]$. The ideal generated by $I(\ell^\prime)$ in the latter
ring, will be denoted by $\hat{I}(\ell^\prime)$.

\begin{lemma} \label{lem:DetII}
Given weakly increasing sequences $i$ and $\ell$ as in \eqref{eq:Deti}
and \eqref{eq:Detl}. Let $a \leq c \leq b$. 
\begin{itemize}
\item[a.] $\hat{L}(i_a, \ldots, i_c) \sus L(i_a,\ldots,i_b) \sus
\hat{L}(i_a,\ldots,i_c) + \hat{L}(i_c,\ldots,i_b)$
\item[b.] $\hat{I}(l_a, \ldots, l_c) \sus I(l_a,\ldots,l_b) \sus
\hat{I}(l_a,\ldots,l_c) + \hat{I}(l_c,\ldots,l_b)$
\end{itemize}
\end{lemma}

\begin{proof}
The first inclusions are immediate from the definitions of these
ideals. We then prove the second inclusions.

a. By Proposition \ref{pro:finPaiso} a minimal generator of 
$L(i_a,\ldots,i_b)$ is the monomial associated to 
\[ (p_a,a),(p_{a+1},a+1), \ldots, (p_d,d)\] where $p_s > i_{s+1}$
for $s < d$ (due to minimality of the generator) and
$p_d \leq i_{d+1}$. If $d < c$, this monomial is in $L(i_a,\ldots,i_c)$.
If $d \geq c$ then the monomial associated to $(p_c,c),\ldots, (p_d,d)$
is in $L(i_c,\ldots,i_b)$.

\medskip
b. Suppose we have a
$(d-a)$-determinant in $J_a^d$ where $d \leq b$. 
If $d \leq c$ then clearly this determinant is in the first summand.
Suppose $d > c$. Let the columns of the determinant be in
positions $q_{a+1}, \ldots, q_d$. Since this determinant is in $J_a^d$ we
have $q_d \leq l_d$. Then by expanding the determinant by the upper
$(d-c)$ rows we see that the determinant is in the second summand.
\end{proof}

\begin{lemma} \label{lem:DetVary}
Suppose $l_c < l_{c+1}$ and let $i$ be associated to $\ell$.
Then the variable $y_{l_c +1,c}$ does not occur
in a minimal  generator of $L^Y(i)$. 
\end{lemma}
\begin{proof} If there is a previous maximum $l_d + 1 -d$ in the sequence
\eqref{eq:DetmaxMax} at or before $l_c + 1 -c$, let $d$ be maximal such. 
Then $d \leq c$ and 
and $l_d +1 - d \geq l_c +1 -c$. Then $l^\prime_d = l_d$ and 
$i_d = l_d +1 -d$. But then there is no variable $x_{i_d,d}$ and so 
no variable $x_{l_c +1 -c, c}$ and so no variable $y_{l_c+1,c}$ in $L^Y(i)$. 

If there is no previous maximum in the sequence \eqref{eq:DetmaxMax} before
$l_c +1 - c$, then $l_a -a \geq l_c +1 -c$ and an analog argument implies
that there is no $y_{l_c +1,c}$ variable. 
\end{proof}


Given $\ell$ with $l_c < l_{c+1}$. Let 
\begin{align*} 
\ell^\prime :\,\, &  l_a \leq \cdots \leq l_c +1 \leq l_{c+1} \leq \cdots \leq l_b.\\
\ell_c :\,\,  & l_a  \leq \cdots \leq l_c \\
\ell^\prime_{c} : \,\, & l_a  \leq \cdots \leq l_c+1 \\
\ell^{c} :\,\,  & l_c + 1 \leq \cdots \leq l_b
\end{align*}

Let $i$ be associated to $\ell$. We then correspondingly have 
$i^\prime, i_c, i^\prime_c, i^c$. 
The ring $\kr[Y;\ell^\prime]$ is $\kr[Y;\ell]/(y_{l_c +1,c})$. Let 
$\oI(\ell)$ be the image of $I(\ell)$ in the latter ring. Also let
$\hI(\ell_c)$ be the ideal generated by $I(\ell_c)$ in this ring, and
so on.

\begin{lemma} \label{lem:DetIrel}
\begin{itemize}
\item[a.] $\oI(\ell) \sus \hI(\ell_c) + \hI(\ell^c)$
\item[b.] $I(\ell^\prime) = \oI(\ell) + \hI(\ell^\prime_c)$
\item[c.] $\hI(\ell^\prime_c) \cdot \hI(\ell^c) \sus \oI(\ell)$
\end{itemize}
\end{lemma}

\begin{proof} 
a. By Lemma \ref{lem:DetII}.b we have 
\[ I(\ell) \sus \hI(\ell_c) + \hI(\ell^c) + (y_{l_c+1,c}) \]
in $\kr[Y;\ell]$. Then map this down to $\kr[Y;\ell^\prime]$.

b. Given a $(d-a)$-determinant in $J_a^d$ of $I(\ell^\prime)$ and let
$q_{a+1}, \ldots, q_d$ be the columns involved in this determinant. If
$d \leq c$ then $q_d \leq l_c +1$ and this determinant is in $\hI(\ell^\prime_c)$.
If $d > c$ then any such determinant of $I(\ell^\prime)$ and $I(\ell)$ are the
same modulo $y_{l_c + 1,c}$. 

c. Let $E$ and $F$ be determinants in the first and second factor.
Let $E$ be an $(e-a)$-determinant of $I(\ell^\prime_c)$ involving 
columns $q_{a+1}, \cdots, q_e$. So  $e \leq c$ and $q_e \leq l_c +1$. 
If $q_e \leq l_c$ then $E$ is 
in $I(\ell)$ and so in $\oI(\ell)$. Suppose $q_e = l_c+1$.
Then we are considering a $(c-a)$-determinant
(by the definition of the ideals in Subsection \ref{subsec:DetmaxDefi}), 
so $e = c$.
Let the determinant $F$ have columns in positions $q_{c+1}, \ldots, q_d$
where $l_c+1 < q_{c+1}$. Consider now the determinant with columns
$q_{a+1}, \ldots, q_c, \ldots, q_d$ in $M(\ell)$. We expand it by the upper
$(d-c)$ rows. Then this determinant is $E \cdot F$ modulo $y_{l_c +1,c}$ and
so $E \cdot F \in \oI(\ell)$.
\end{proof}

\begin{proof}[Proof of Theorem \ref{thm:DetMain}.]
We are going to use induction on $\ell$.
The induction start will
be as in Example \ref{ex:DetIndstart}, using the sequence
\[ \ell \, : \, l_a = \cdots =  l_{b-1} \leq l_b. \]
Then $M(\ell)$ is a $(b-a) \times (l_b - l_a)$ matrix $M(\ell)$ of distinct
variables, and the statement is classic, \cite{Stu}.
We will successively increase the $l_i$ for $i < b$. We then successively
replace variables by zero in the matrix $M(\ell)$.

By assumption we have $\text{in}(I(\ell)) = L^Y(i)$ for some $\ell$ and
associated $i$. We want to show
that $\text{in}(I(\ell^\prime)) = L^Y(i^\prime)$. 
Denote by $\overline{L}^Y(i)$ the image of $L^Y(i)$ in 
$\kr[Y;\ell^\prime] = \kr[Y;\ell]/(y_{\ell_c+1,c})$. 
We note that
$L^Y(i^\prime) = \overline{L}^Y(i) + \hat{L}^Y(i_c^\prime)$.
By Lemma \ref{lem:DetIrel}b. 
$I(\ell^\prime) = \overline{I}(\ell) + \hat{I}(\ell_c^\prime)$.
What we need to show is then
\begin{equation*} 
\text{in}(\oI(\ell) + \hI(\ell^\prime_c)) = 
\overline{L}^Y(i) + \hat{L}^Y(i^\prime_c). 
\end{equation*} 
By induction the following inclusion is clear
\begin{equation} \label{eq:DetInideal}
 \text{in}(\oI(\ell) + \hI(\ell^\prime_c)) \supseteq
\overline{L}^Y(i) + \hat{L}^Y(i^\prime_c). 
\end{equation}

\medskip
Consider the exact sequence
\begin{equation} \label{eq:DetSeqI}
0 \vpil \frac{\kr[Y;\ell^\prime]}{\hI(\ell^\prime_c) + \hI(\ell^c)} \vpil
\frac{\kr[Y;\ell^\prime]}{\hI(\ell_c) + \hI(\ell^c)} 
\overset{\alpha}{\longleftarrow} \oI(\ell) + \hI(\ell^\prime_c).
\end{equation}
We show that the kernel of $\alpha$ is $\oI(\ell)$. That the latter is in
the kernel follows by Lemma \ref{lem:DetIrel}.a. Let then 
$f^\prime_c \in \hI(\ell^\prime_c)$ be in the kernel of $\alpha$, so
$f^\prime_c = f_c + f^c$ where $f_c \in \hI(\ell_c)$ (and so
is in $\oI(\ell)$ by Lemma \ref{lem:DetII}b.), and $f^c \in \hI(\ell^c)$.
Then $f^\prime_c - f_c = f^c$. But the left side is in the ideal 
in $\kr[Y;\ell^\prime]$ generated
by $I(\ell^\prime_c) \sus \kr[Y;\ell^\prime_c]$, and $f^c$ is in the ideal
generated by $I(\ell^c) \sus \kr[Y;\ell^c]$. These polynomial rings
have distinct variables. Then 
\[ \hI(\ell^\prime_c) \cap \hI(\ell^c) = \hI(\ell^\prime_c) \cdot \hI(\ell^c). \]
Thus 
\[ f^\prime_c - f_c \in \hI(\ell^\prime_c) \cdot \hI(\ell^c) \sus
\oI(\ell) \quad \text{ by Lemma \ref{lem:DetIrel} }.\]
Hence $f^\prime_c \in \oI(\ell)$.

\medskip
There is also an exact sequence
\begin{equation} \label{eq:DetSeqL}
0 \vpil \frac{\kr[Y;\ell^\prime]}{\hat{L}^Y(i^\prime_c) + \hat{L}^Y(i^c)} \vpil
\frac{\kr[Y;\ell^\prime]}{\hat{L}^Y(i_c) + \hat{L}^Y(i^c)} 
\vpil
\overline{L}^Y(i) + \hat{L}^Y(i^\prime_c) \vpil \overline{L}^Y(i)
\vpil 0.
\end{equation} 

We now compare the Hilbert functions of the terms in the sequences 
\eqref{eq:DetSeqI} and \eqref{eq:DetSeqL}.

\noindent 1. By induction $L^Y(i_c)$ is the initial ideal of $I(\ell_c)$ and so
these ideals have the same Hilbert functions. The same also holds
true for the pairs $i^\prime_c, \ell^\prime_c$, for $i^c, \ell^c$ and for $i,\ell$.

\noindent 2. Since $L^Y(i^\prime_c)$ and $L^Y(i^c)$ involve distinct 
sets of variables,
the Hilbert functions of $\hat{L}^Y(i^\prime_c) + \hat{L}^Y(i^c)$ and 
$\hI(\ell^\prime_c) + \hI(\ell^c)$ are the same, and similarly with
$i^\prime_c, \ell^\prime_c$ replaced by $i_c,\ell_c$. 

\noindent 3. By Lemma \ref{lem:DetVary} $(L^Y(i):y_{\ell_c +1,c}) = L^Y(i)$.
Since $\text{in}(I(\ell)) = L^Y(i)$ it is an easy fact that
$(I(\ell):y_{\ell_c+1,c}) = I(\ell)$. Hence the Hilbert functions
of $\overline{L}^Y(i)$ and $\overline{I}(\ell)$ in
$\kr[Y;\ell^\prime] = \kr[Y;\ell]/(y_{\ell_c +1,c})$ are the same.

\noindent 4. Comparing the sequences
\eqref{eq:DetSeqI} and \eqref{eq:DetSeqL} 
\[ \overline{L}^Y(i) + \hat{L}^Y(i^\prime_c), \quad I(\ell^\prime) = \oI(\ell) +
\hI(\ell^\prime_c) \]
have the same Hilbert function. Taking the inclusion 
\eqref{eq:DetInideal} into account, this inclusion must be an equality.
\end{proof}

\section{Strongly stable ideals}
\label{sec:ststable}

We now assume $P$ is the totally ordered poset 
\[ P = [m] = \{1 < 2 < \cdots < m\}. \]
Let $\cJ \sus \Hom([m],\NN)$ be a poset ideal.
We shall show that the associated letterplace ideal gives
by projection a strongly stable ideal in $\kr[x_{[m]}]$, and
the co-letterplace ideal gives by projection a strongly stable
ideal in $\kr[x_{\NN}] = \kr[x_0,x_1,x_2, \ldots]$ generated in degree $\leq m$. 
Each of these correspondences are one-to-one
and putting them together we get a duality
between strongly stable ideals in $\kr[x_{[m]}]$ and finitely generated
$m$-regular
strongly stable ideals in $\kr[x_\NN]$. 
The results in this section are joint with
Alessio D'Ali and Amin Nematbakhsh.

If $S = \sum_{p \in P} s_p p$ we shall often by abuse of notation 
write $S$ as a short for the
monomial $\prod_{p \in P} x_p^{s_p}$.

\subsection{Strongly stable ideals from letterplace ideals}
Recall that an ideal $I$ of $\kr[x_1, \ldots, x_n]$ is {\it strongly
stable} if $x_jm \in I$ and $i < j$ implies $x_im \in I$.

The letterplace ideal $L(\cJ,P)$ is an ideal in $\kr[x_{[m] \times \NN}]$.
The projection onto the first factor $[m] \times \NN \mto{p_1} [m]$
has right strict chain fibers, and so we get the projected ideal $L^{p_1}(\cJ,P)
\sus \kr[x_{[m]}]$ such that $\kr[x_{[m]}]/L^{p_1}(\cJ,P)$ is a regular quotient
of $\kr[x_{S}]/L(\cJ,P)$ for suitable finite $S$. 

\begin{theorem}
The projected letterplace ideal $L^{p_1}(\cJ,[m])$ 
is a strongly stable ideal in $\kr[x_{[m]}]$. 
This correspondence is a one-one correspondence between poset ideals
in $\Hom([m],\NN)$ and strongly stable ideals  $\cJ$ in $\kr[x_{[m]}]$. 
\end{theorem}

\begin{proof}
The monomials in $L^{p_1}(\cJ,[m])$ are the images $\oLa\phi$ where
$\phi \in \cJ^c$. For a map $\phi: [m] \pil \NN$ then 
$\Lambda \phi = \{ (a,i) \, | \, \phi(a) > i \geq \phi(a-1) \}$. This gives
$\oLa \phi = \prod_{a = 1}^m x_a^{\phi(a) - \phi(a-1)}$. Suppose 
that $a \geq 1$ and $\phi(a) > \phi(a-1)$ so $x_a$ is a factor of 
$\overline{\Lambda} \phi$. Define 
\[ \phi^\prime(i) = \begin{cases} \phi(i), & i \neq a-1 \\
\phi(a-1) + 1, & i = a-1 \end{cases}. \]
Then $\oLa{\phi^\prime} = \frac{x_{a-1}}{x_a} \oLa \phi$ and
$\phi^\prime \geq \phi$ and so $\phi^\prime \in \cJ^c$. This implies that 
$\oLa\cJ^c = L^{p_1}(\cJ,P)$ is strongly stable.

Since 
\[ \Hom([m], \NN) \mto{\overline{\Lambda}} \text{ monomials in }
\kr[x_1, \ldots, x_m] \]
is a bijection, this gives an injective map from poset ideals $\cJ$ to strongly
stable ideals in $\kr[x_1, \ldots, x_m]$. 

Consider now a strongly stable ideal $I \sus \kr[x_1, \ldots, x_m]$.
The monomials in $I$ correspond by Proposition \ref{pro:regLbar} 
via the inverse of $\oLa$ to a subset
$\cF \sus \Hom([m], \NN)$. We show that $\cF$ is a poset filter. 
Let $\phi \in \cF$ and $\psi > \phi$. 
For some $a$ we must then have $\phi(a) < \psi(a)$. Let $a$ be maximal such.
Define 
\[ \phi^\prime(i) = \begin{cases} \phi(i), & i \neq a \\
\phi(a) + 1, & i = a.
\end{cases} \]
Then $\psi \geq \phi^\prime > \phi$. But $\oLa \phi
= \prod_{a = 1}^m x_a^{\phi(a) - \phi(a-1)}$ and $\oLa \phi^\prime = 
\frac{x_a}{x_{a+1}} \cdot \oLa \phi$ if $a < m$ and 
$\oLa \phi^\prime = x_a \oLa \phi$ if $a = m$. 
Then $\oLa \phi^\prime \in I$ and so $\phi^\prime \in \cF$. 
Continuing we eventually get $\psi \in \cF$. Thus $\cF$ is poset filter
and so $\cF = \cJ^c$ for a poset ideal $\cJ$.
\end{proof}

\subsection{Strongly stable ideals from co-letterplace ideals}

Given the poset ideal $\cJ \sus \Hom([m],\NN)$ we get the co-letterplace
ideal $L([m],\cJ)$. The projection $[m] \times \NN \mto{p_2} \NN$ has
left strict chain fibers. We then get the ideal $L^{p_2}([m], \cJ)$ in 
$\kr[x_{\NN}]$ such that $\kr[x_\NN]/L^{p_2}([m], \cJ)$ is a regular
quotient of $\kr[x_{[m] \times \NN}]/L([m],\cJ)$. 

\begin{theorem} The projected co-letterplace ideal 
$L^{p_2}([m],\cJ)$ is a strongly stable $m$-regular ideal in $\kr[x_\NN]$.
This gives a one-one correspondence between poset ideals 
$\cJ \sus \Hom([m], \NN)$ and finitely generated strongly stable 
$m$-regular ideals in $\kr[x_\NN]$. 
\end{theorem}

\begin{proof}
Let $I = [1,i] \sus [m]$ be a poset ideal and $\phi : [1,i] \pil \NN$
be a marker for $\cJ$. 
Then $\Gamma \phi = \{ (a,j) \, | \, \phi(a) = j, 1 \leq a \leq i \}$ and 
$\overline{\Gamma \phi}$ in $\kr[x_\NN]$ is $\prod_{a=1}^i x_{\phi(a)}$. 
Take a variable $x_j$ with $j\geq 1$ and dividing $\overline{\Gamma \phi}$.
Let $a$ be minimal with $\phi(a) = j$. Define
\[ \phi^\prime(i) = \begin{cases} \phi(i), & i \neq a \\
\phi(a) - 1 = j-1, & i = a. \end{cases} \]
Then $\phi^\prime < \phi$ and so $\phi^\prime \in \cJ$ and
$\overline{\Gamma \phi^\prime} = \frac{x_{j-1}}{x_j} 
\cdot \overline{\Gamma \phi} \in L^{p_2}([m],\cJ)$. 
Then $L^{p_2}([m], \cJ)$ is strongly stable.

\medskip
Conversely, given a finitely generated $m$-regular strongly stable ideal
$I \sus \kr[x_\NN]$. Each generator $x_{i_1} x_{i_2} \cdots x_{i_r}$
where $i_1 \leq i_2 \leq \cdots \leq i_r$ will have $r \leq m$ and
gives an isotone map $\alpha : [1,r] \pil \NN$. Let $\cJ$ be the set of isotone
map $\phi: [m] \pil \NN$ extending such $\alpha$'s. We claim that
$\cJ$ is a poset ideal. 

For this it is sufficient to show that if $\beta \leq \alpha$ then 
$\overline{\Gamma \beta}$ is in $I$. But $\overline{\Gamma \beta} 
= x_{j_1} x_{j_2} \cdots x_{j_r}$ where each
$j_k = \beta(k) \leq \alpha(k) = i_k$. Then by strong stability of $I$
we see that $\overline{\Gamma \beta}$ is in $I$ and so $\beta$ is in $\cJ$.
\end{proof}

\begin{corollary}
There are one-to-one correspondences:
\begin{align*} & \text{Strongly stable ideals in } \kr[x_{[m]}] \\
\overset{1-1} {\longleftrightarrow}  & 
\text{ Poset ideals } \cJ \sus \Hom([m],\NN) \\
\overset{1-1}{\longleftrightarrow}  & \text{ Finitely generated 
strongly stable $m$-regular ideals in } \kr[x_\NN].
\end{align*}
\end{corollary}

\begin{example}
Let $m = 1$. The strongly stable ideal $(x_1^n) \sus \kr[x_1]$ corresponds
to the $1$-regular ideal $(x_0,x_1,x_2, \ldots, x_{n-1}) \sus \kr[x_\NN]$. 

Let $m = 2$. The strongly stable ideal 
\[ (x_1^a, x_1^{a-1}x_2^{b_1 + 1}, \ldots, x_1^{a-r}x_2^{b_r + r}, \ldots,
x_2^{b_a + a})\]
with $0 \leq b_1 \leq b_2 \leq \cdots \leq b_a$ corresponds to 
the smallest $2$-regular strongly stable ideal containing
\[ x_0 x_{b_a + a-1}, x_1x_{b_{a-1} + a-1}, \cdots, x_{a-1}x_{b_1 + a-1}. \]
\end{example}

In the following let $[n]_0 = \{0 < 1 < \cdots < n\}$ be the chain with 
$n+1$ elements. By the adjunction between $P \times -$ and
$\Hom(P,-)$ in the category of posets, we have 
\begin{align*} 
\Hom([m],[n]_0) = \Hom([m],\Hom([n],[1]_0)) = & \Hom([m] \times [n], [1]_0)) \\
= & \Hom([n], \Hom([m],[1]_0) \\
= & \Hom([n],[m]_0).
\end{align*}
This is the correspondence between a partition of $m$ parts into sizes $\leq n$,
and its dual partition of $n$ parts into sizes $\leq m$. We then get
the following.

\begin{corollary}
There are one-to-one correspondences:
\begin{align*} & \text{ Strongly stable $n$-regular ideals in } \kr[x_{[m]}] \\ 
\overset{1-1}{\longleftrightarrow} & \text{ Poset ideals } 
\cJ \sus \Hom([m],[n]_0) \\
\overset{1-1}{\longleftrightarrow} & \text{
Strongly stable $m$-regular ideals in } \kr[x_{[n]}].
\end{align*}
\end{corollary}

\section{Proof that the poset maps induce regular 
sequences}
\label{sec:proof}
The following is a refinement of Lemma 7.1 of \cite{FGH}, but
the proof is exactly the same, and we omit it.

\begin{lemma} \label{lem:regx01}
Let $I \sus k[x_0,\ldots,x_n]$ be a monomial ideal such that every minimal 
generator of $I$ is squarefree in the variable $x_0$, so no minimal generator
is divisible by $x_0^2$. If $f \in S$ is such that $x_0f = x_1f$ in 
$\kr[x_0, \ldots, x_n]/I$, then 
for every monomial $m$ in $f$ we have $x_1m = 0 = x_0 m$ in this
quotient ring.
\end{lemma}

\begin{proof}[Proof of Theorem \ref{thm:regLP}.]
Recall by Lemma \ref{lem:regFilters} 
that an isotone map $\phi: P \pil \NN$ corresponds to a chain of 
poset filters in $P$:
\[ P = F_0 \supseteq F_1 \supseteq \cdots \supseteq 
F_n \supseteq F_{n+1} =
\emptyset, \]
where $\phi(p) = i$ if $p \in F_i \backslash F_{i+1}$. 
Then $\Lambda \phi$ is the set of all pairs $(a,i-1)$ where $a$ is a minimal
element in $F_i$ for $i \geq 1$.

\medskip
Let $\phi^{-1}(r)$ have cardinality $\geq 2$. Let $(p,i)$ be the element
in the fiber with minimal $i$ and so $\phi^{-1}(r) = R_1 \cup R_2
= \{(p,i)\} \cup R_2$ is a disjoint union. 
By Lemma 8.1 in \cite{FGH} we have a factorization into isotone maps of posets 
\begin{equation} \label{eq:ProofREF}
\xymatrix{ S \ar[rr] \ar[dr]^{\phi^\prime} & & R \\
  & R^\prime \ar[ur]^{\eta} & }
\end{equation}
where the cardinality of $R^\prime$ is one more than that of $R$,
$\eta^{-1}(r) = \{r_1, r_2 \}$, and
${\phi^{\prime}}^{-1}(r_1) =\{(p,i)\}$ and ${\phi^\prime}^{-1}(r_2) = R_2$.

For $(p,i) \in S$ denote by $\ov{p,i}$ its image in $R^\prime$ and
by $\ovv{p,i}$ its image in $R$. Note that all minimal generators 
of $L^{\phi^\prime}(\cJ,P)$ are squarefree with respect to 
$x_{r_1} = x_{\ov{p,i}}$. We need to show that $x_{r_1} - x_{r_2}$ is a 
non-zero divisor, so $x_{r_1} f = x_{r_2} f$ implies $f = 0$. By 
Lemma \ref{lem:regx01} it is enough to show that for any monomial $m$,
then $x_{r_1}m = 0 = x_{r_2}m$ in $\kr[x_{R^\prime}]/L^{\phi^\prime}(\cJ,P)$
implies $m = 0$ in this quotient ring.

\noindent {\bf Note.} In \cite[Theorem 2.1]{FGH} 
the proof given in Section 8 there had a minor gap, in that Lemma 7.1 in
\cite{FGH} was not quite sufficient to conclude as above that $m = 0$. However
Lemma \ref{lem:regx01} above rectifies this.

So assume $m$ is nonzero. We will derive a contradiction. Let $(p,i)$
map to $r_1$ and $(q,j)$ map to $r_2$.  For $\phi$ inclusion minimal
in the complement
$\cJ^c$, the ascent $\Lambda \phi$ lives
in $\NN S$, and denote by $\overline{\Lambda \phi}$ its image in the monoid
$\NN R^\prime$. There are then $\phi, \psi \in \cJ^c$ such that 
$m_{\overline{\Lambda \phi}}$ divides $x_{\ov{p,i}}m$ and 
$m_{\overline{\Lambda \psi}}$ divides $x_{\ov{q,j}}m$. 
Let $\phi$ and $\psi$ correspond to respectively
\[ P = F_0 \supseteq F_1 \supseteq \cdots, \quad 
P = G_0 \supseteq G_1 \supseteq \cdots. \]
Then there is $(a_s, s-1)$ with $a_s \in \min F_s$ such that 
$\overline{(a_s,s-1)} = \overline{(p,i)}$, and $(b_t, t-1)$ with 
$b_t \in \min G_t$ and $\overline{(b_t,t-1)} = \overline{(q,j)}$. 

Since $\ovv{a_s,s-1} = \ovv{b_t,t-1}$ and $\phi$ has right strict chain 
fibers, we have, say $s < t$ and $a_s \geq b_t$. 

Let $A_s$ be the minimal elements of $F_s$ and let $F_s^\prime$ be the filter
generated by $A_s \backslash \{a_s \}$. 
We claim that $F_s^\prime \cup G_s =  F_s \cup G_s$.
Let $p \in F_s$. If $p \geq $ some element in $A_s \backslash \{ a_s \}$
then clearly $p \in F_s^\prime$. If $p \geq a_s$ then $p \geq b_t$. But
since $t > s$ then $p \in G_s$. 
Now consider the following sequence of poset filters:
\begin{equation} \label{lig:RegChainIJ}
P = F_0 \cup G_0 \supseteq F_1 \cup G_1 \supseteq \cdots
\supseteq F_{s-1} \cup G_{s-1} \supseteq F_s^\prime \cup G_s 
\supseteq F_{s+1} \supseteq \cdots .
\end{equation}
This chain corresponds to an isotone map $\phi^\prime$ where
$\phi^\prime \geq \phi$ (since for each $p$ the $p$'th term above 
contains $F_p$), and so $\phi^\prime \in \cJ^c$. 

Now for two poset ideals $F, G \sus P$, we have $\min (F \cup G) \sus
(\min F) \cup (\min G)$. Thus for $i \in \NN$ 
the product $\prod_{p \in \min(F \cup G)} x_{(p,i-1)}$ divides
the least common multiple of $\prod_{p \in \min F} x_{(p,i-1)}$
and $\prod_{p \in \min G} x_{(p,i-1)}$.
But all the variables occurring in each of these monomials
will by $S \mto{\phi^\prime} 
R^\prime$ map to distinct variables, since $\phi$ has right strict
chain fibers. Thus we will also have that 
the product $\prod_{p \in \min(F \cup G)} x_{\ov{p,i-1}}$ divides
the least common multiple of $\prod_{p \in \min F} x_{\ov{p,i-1}}$
and $\prod_{p \in \min G} x_{\ov{p,i-1}}$.
But then $m_{\overline{\Lambda \phi^\prime}}$ constructed from the
chain \eqref{lig:RegChainIJ}, see Lemma \ref{lem:regFilters}, 
divides the least common multiple of 
$m_{\overline{\Lambda \phi}}/x_{\ov{a_s,s-1}}$ and 
$m_{\overline{\Lambda \psi}}/x_{\ov{b_t,t-1}}$ which both divide $m$. 
Hence $m_{\overline{\Lambda \phi^\prime}}$ divides $m$, contradicting
that $m$ is nonzero in $\kr[x_{R^\prime}]/L^{\phi^\prime}(\cJ,P)$.
\end{proof}

\begin{proof}[Proof of Theorem \ref{thm:regCOLP}.]
By induction on the cardinality of $\im \phi$. We may assume we have a
factorization 
\begin{equation}
\xymatrix{ S \ar[rr] \ar[dr]^{\phi^\prime} & & R \\
  & R^\prime \ar[ur]^{\eta} & }
\end{equation}
analogous to 
\eqref{eq:ProofREF}, with $\phi^{-1}(r)$ of cardinality $\geq 2$ and
with $\eta^{-1}(r) = \{r_1,r_2 \}$ and $\phi^{\prime -1}(r_1) 
= R_1 = \{(p_0,a) \}$ and $\phi^{\prime -1}(r_2) = R_2$. Furthermore
we have by induction established that
\begin{equation} \label{lig:regRprim}
\kr[x_{R^\prime}]/ L^{\phi^\prime}(P,\cJ)
\end{equation}
is obtained by cutting down
from $\kr[x_{S}]/L(P,\cJ)$ by a regular sequence
of variable differences. 

Let $(p_0,a)$ in $S$ map to $r_1 \in R^\prime$ and $(q_0,b)$ map to 
$r_2 \in R^\prime$. We will show that $x_{r_1} - x_{r_2}$ is a regular element in
the quotient ring 
(\ref{lig:regRprim}). 
So let $f$ be a polynomial of this quotient ring such that
$f(x_{r_1} - x_{r_2}) = 0$. 
Then by Lemma \ref{lem:regx01}, for any monomial $m$ in $f$
we have $m x_{r_1} = 0 = m x_{r_2}$ in the quotient ring 
$\kr[x_{R^\prime}]/ L^{\phi^\prime}(P,\cJ)$.
We assume $m$ is nonzero in this quotient ring
and shall derive a contradiction.

There is a minimal marker $i : I \pil \NN$  for $\cJ \sus \Hom(P, \NN)$ 
such that the
monomial $m^i = \prod_{p \in I} x_{\ov{p,i_p}}$ in 
$L^{\phi^\prime}(P,\cJ)$ divides $m x_{\ov{p_0,a}}$,
and similarly a minimal marker $j : J \pil \NN$ such that the monomial
$m^j = \prod_{p \in J} x_{\ov{p,j_p}}$ divides $m x_{\ov{q_0,b}}$.
Hence there are $s$ and $t$ in $P$ such that
$\ov{s,i_s} = \ov{p_0,a}$ and $\ov{t,j_t} = \ov{q_0,b}$. 
In $R$ we then get:
\[ \ovv{s,i_s} = \ovv{p_0,a} = \ovv{q_0,b} = \ovv{t,j_t}, \]
so $s = t$ would imply $i_t = j_t$ since $\phi$ has left strict chain fibers.
But then
\[ r_1 = \ov{p_0,a} = \ov{s,i_s} = \ov{t,j_t} = \ov{q_0,b} = r_2 \]
which is not so.
Assume then, say $s < t$.
Then $i_s \geq j_t$ since $\phi$ has left strict chain fibers, and so
\begin{equation}
\label{lig:regProofst}
i_t \geq i_s \geq j_t \geq j_s.
\end{equation}

In the following let $i_p = \infty$ if $p \not \in I$ and 
similarly $j_p = \infty$ if $p \not \in J$.
Form the monomials
\begin{itemize}
\item $m^i_{> s} =  \underset{p \in I, p > s} \prod x_{\ov{p,i_p}}$.
\item $m^i_{i > j} = \underset{p \in I, i_p > j_p, not \, (p > s)}{\prod} x_{\ov{p, i_p}}$.
\item $m^i_{i < j} = \underset{p \in I, i_p < j_p, not \, (p > s)}{\prod} x_{\ov{p, i_p}}$.
\item $m^i_{i = j} = \underset{p \in I, i_p = j_p, not \, (p > s)}{\prod} x_{\ov{p, i_p}}$.
\end{itemize}
Similarly we define $m^j_*$ for the various subscripts $*$. 
Then 
\[ m^i = m^i_{i=j} \cdot m^i_{i > j} \cdot m^i_{i < j} \cdot m^i_{>s} \]
divides $x_{\ov{s, i_s}} m$, and
\[ m^j = m^j_{i=j} \cdot m^j_{i > j} \cdot m^j_{i < j} \cdot m^j_{>s} \]
divides $x_{\ov{t, j_t}} m$.

Now let 
\[ \tm^j_{i > j}
= \underset{p \in I \cap J, i_p > j_p, not \, (p > s)}{\prod} x_{\ov{p, j_p}}, \]
which is the factor of $m^j_{i>j}$ where we take the product only over
$I \cap J$ and not over $J$.

There is now a map $\ell : I \pil \NN$ defined by
\[ \ell(p) = \begin{cases} i_p \, \mbox{ for } p \in I, p > s \\
                       \min(i_p, j_p) \, \mbox{ for } p \in I 
\mbox{ and not } (p > s) \\
\text{(recall that $j_p = \infty$ if $p \not \in J$)}
          \end{cases}.
\]
This is an isotone map as is easily checked. Its associated monomial
is 
\begin{equation} \label{eq:proof-ml}
 m^\ell = m_{i=j} \cdot \tm^j_{i > j} \cdot m^i_{i < j} \cdot m^i_{>s}. 
\end{equation}
We will show that this divides $m$. Since the marker $\ell$ is
$\leq$ the marker $i$, this will prove the theorem.

\medskip
\begin{claim} $\tm^j_{i> j}$ is relatively prime to 
$m^i_{i<j}$ and $m^i_{>s}$. 
\end{claim}

\begin{proof}
Let $x_{\ov{p, j_p}}$ be in $\tm^j_{i > j}$. 

1. Suppose it equals the variable
$x_{\ov{q, i_q}}$ in $m^i_{ i < j}$. Then $p$ and $q$ are comparable
since $\phi$ has left strict chain fibers.
If $p < q$ then $j_p \geq i_q \geq i_p$, contradicting $i_p > j_p$. 
If $q < p$ then $i_q \geq j_p \geq j_q$ contradicting $i_q < j_q$.

2 Suppose $x_{\ov{p, j_p}}$ equals $x_{\ov{q, i_q}}$ in $m^i_{> s}$. 
Then $p$ and $q$ are comparable and so $p < q$ since $q > s$ and we do not
have $p > s$. Then $j_p \geq i_q \geq i_p$ contradicting $i_p > j_p$. 
\end{proof}

\begin{claim} \label{claim:proofLcm}
$m_\ell$ divides $m x_{\ov{s,i_s}}$.
\end{claim}

\begin{proof}
Let $abc =  m^i_{i=j} \cdot m^i_{i < j} \cdot m^i_{>s}$ which divides
$m x_{\ov{s,i_s}}$ and $a b^\prime = m^j_{i=j} \cdot \tm^j_{i > j}$
which divides $m$ since $x_{\ov{t,j_t}}$ is a factor of $m^j_{>s}$
since $t > s$. 
Now if the product of monomials $abc$ divides the monomial $n$ 
and $ab^\prime$ also divides $n$, and $b^\prime$ is relatively prime to $bc$,
then the least common multiple $abb^\prime c$ divides $n$.
We thus see that the monomial associated to the isotone map $\ell$
\[ m^\ell = m_{i=j} \cdot \tm^j_{i > j} \cdot m^i_{i < j} \cdot m^i_{>s} \]
divides $m x_{\ov{s,i_s}}$.
\end{proof}
 
We need now only show that the variable $x_{\ov{s,i_s}}$ occurs to
a power in the above  product \eqref{eq:proof-ml} 
for $m^\ell$ less than or equal to that of its
power in $m$. 

\begin{claim} $x_{\ov{s,i_s}}$ is not a factor of
$\tm^j_{i > j}$ or $m^i_{i < j}$. 
\end{claim}

\begin{proof}
1. Suppose $\ov{s, i_s} = \ov{p, i_p}$ where $i_p < j_p$ 
and not $p > s$. Since $p$ and $s$ are comparable 
(they are both in a fiber of $\phi$), we have $p \leq s$.
Since $\phi$ is isotone $i_p \leq i_s$ and since $\phi$ has
left strict chain fibers $i_p \geq i_s$. Hence $i_p = i_s$.
By (\ref{lig:regProofst}) $j_s \leq i_s$ and so 
$j_p \leq j_s \leq i_s = i_p$. This contradicts $i_p < j_p$. 

2. Suppose $\ov{s,i_s} = \ov{p, j_p}$ where $j_p < i_p$ and not $p > s$.
Then again $p \leq s$ and $i_p \leq i_s \leq j_p$, giving a contradiction.
\end{proof}

If now $i_s > j_s$ then $x_{\ov{s, i_s}}$ is a factor in
$m^i_{i>j}$ but by the above, not in $\tm^j_{i>j}$. Since $m^\ell$ is obtained from 
$m^i$ by replacing
$m^i_{i > j}$ with $\tm^j_{i > j}$, we see that $m^\ell$ contains
a lower power of $x_{\ov{s,i_s}}$ than $m^i$ and so $m^\ell$ divides $m$.


\begin{claim} \label{ProofClaimLik} Suppose $i_s = j_s$. Then the power of
$x_{\ov{s,i_s}}$ in $m^i_{ >s}$ is less than or equal to its power
in $m^j_{> s}$. 
\end{claim}

\begin{proof}
Suppose $\ov{s, i_s} = \ov{p, i_p}$ where $p > s$. 
We will show that then $i_p = j_p$. This will prove the claim.

The above implies $\ovv{p, i_p} = \ovv{s, i_s} = \ovv{t, j_t}$,
so either $s < p < t$ or $s < t \leq p$.
If the latter holds, then since $\phi$ has left strict chain fibers,
$i_s \geq j_t \geq i_p$ and also $i_s \leq i_p$ by isotonicity, 
and so $i_s = i_p = j_t$. Thus
\[ \ov{s,i_s} \leq \ov{t, j_t} \leq \ov{p, i_p} \]
and since the extremes are equal, all three are equal contradicting
the assumption that the two first are unequal.

Hence $s < p < t$. By assumption on the fibre of $\phi$ we have 
$i_s \geq i_p$ and by isotonicity $i_s \leq i_p$ and so $i_s = i_p$.
Also by (\ref{lig:regProofst}) and isotonicity
\[ i_s \geq j_t \geq j_p \geq j_s .\]
By assumption $i_s = j_s$, and we get equalities everywhere and so $i_p = j_p$,
as we wanted to prove. 
\end{proof}

By \eqref{lig:regProofst} we know that $i_s \geq j_s$. 
In case $i_s > j_s$ we have shown before Claim \ref{ProofClaimLik} 
that $m^\ell$ divides $m$.
So suppose $i_s = j_s$. By the above two claims, the $x_{\ov{s,i_s}}$ 
in $m^\ell$ occurs only in $m_{i=j} \cdot m^i_{ > s}$ and to 
a power less than or equal to that in $m_{i = j} \cdot m^j_{>s}$. 
Since $m^j$ divides $m x_{\ov{t,j_t}}$ and 
$\ov{s,i_s} \neq \ov{t, j_t}$ the power of $x_{\ov{s,i_s}}$
in $m^j$ is less than or equal to its power in $m$. 
Hence the power of $x_{\ov{s,i_s}}$ in $m^\ell$ is less or
equal to its power in $m$ and so by Claim \ref{claim:proofLcm} 
$m^\ell$ divides $m$.

\end{proof}

\bibliographystyle{amsplain}
\bibliography{Bibliography}

\end{document}